\documentclass[12pt,twoside]{amsart}

  \usepackage{url}
  \usepackage{amssymb,amsmath}
     \usepackage{graphicx}

   \usepackage[a4paper, margin=2.5cm]{geometry}

  \input xypic
  \xyoption{all}

 \usepackage{tikz}
\usetikzlibrary{arrows}

 \usepackage{times}

  \usepackage{amsmath, amsthm, amsfonts}

\newtheorem{theorem}{Theorem}[section]
\newtheorem{lemma}[theorem]{Lemma}
\newtheorem{proposition}[theorem]{Proposition}

\theoremstyle{definition}
\newtheorem{definition}[theorem]{Definition}

\theoremstyle{remark}
\newtheorem{remark}[theorem]{Remark}

\numberwithin{equation}{section}

  \newcommand{\cJ}{{\mathcal J }}

  \newcommand{\cF}{{\mathcal F}}
  \newcommand{\cE}{{\mathcal E}}
  
   \newcommand{\cN}{{\mathcal N}}
  
  \renewcommand{\cD}{{\mathcal D}}
  
  \newcommand{\cU}{{\mathcal U}}

  \newcommand{\sfG}{{\mathcal G }}
  \newcommand{\cC}{{\mathcal C }}

\newcommand{\cG}{{\mathcal G }}
\newcommand{\cT}{{\mathcal T }}

\newcommand{\cS}{{\mathcal S }}

  \newcommand{\bU}{{\mathbf U }}
   \newcommand{\diff}{{\mathbf {Diff} }}

   \newcommand{\ba}{\begin{eqnarray}}
   \newcommand{\na}{\end{eqnarray}}

\newcommand{\Gl}{{\mathbf  {Gl}}}

  \newcommand{\C}{{\mathbb C}}
  \newcommand{\R}{{\mathbb R}}
  \newcommand{\Z}{{\mathbb Z}}

\newcommand{\K}{{\mathbb K}}

   \newcommand{\Mod}{\mathrm{Mod}}
  \renewcommand{\a}{\alpha}
  \renewcommand{\b}{\beta}
   
  \renewcommand{\c}{\gamma}

\def \sfM{\mathsf{M}}
\def \sfGl{\mathsf{Gl}}

  \newcommand{\nin}{\noindent}

\def \mc{\mathcal}

\begin{document}

\title{Gluing principle  for orbifold stratified spaces}

  \author{Bohui Chen, An-Min Li}
  \address{School of mathematics\\
  Sichuan University\\
 Chengdu, China}
  \email{bohui@cs.wisc.edu \\
  math$\_$li@yahoo.com.cn}

  \author{Bai-Ling Wang}
  \address{Department of Mathematics\\
  Australian National University\\
  Canberra ACT 0200 \\
  Australia}
  \email{bai-ling.wang@anu.edu.au}

\date{}



   \begin{abstract}   In this paper, we  explore the theme of orbifold stratified spaces and    establish a general criterion for them to be  smooth orbifolds. This criterion utilizes  the notion of linear stratification on the gluing bundles for the   orbifold stratified spaces. We introduce a concept of  good gluing structure to ensure a smooth structure on the stratified space.    As an application, we provide an  orbifold  structure on the coarse moduli space  $\overline{M}_{g, n}$ of stable genus $g$  curves with $n$-marked points.
   Using  the gluing theory  for   $\overline{M}_{g, n} $ associated to   horocycle structures, there is  a  natural  
   orbifold gluing atlas on  $\overline{M}_{g, n} $. We  show this gluing atlas can be refined to provide  a good orbifold  gluing structure and hence a  smooth orbifold structure on $\overline{M}_{g,n}$.  This general gluing principle will be very useful  in  the study of the gluing theory for the compactified moduli spaces of stable pseudo-holomorphic curves in a symplectic manifold.     
    \end{abstract}

   \maketitle

  \tableofcontents

\section{Introduction and statements of main theorems}

Assuming that we have the {\em transversality}  for each moduli space of stable maps in a symplectic  manifold 
$(X, \omega)$ with a domain of a fixed topological type, the
 compactified  moduli space $\overline{M}_{g, n}(X, A, \omega, J)$ of stable maps in a symplectic manifold   $(X, \omega)$ is usually a compact Hausdorff space stratified by smooth orbifolds for a compatible almost complex structure $J$ and $A\in H_2(X, \Z)$.  This moduli space is an example of the so-called     {\em  orbifold stratified spaces}  in this paper, which  is a disjoint union of locally closed smooth orbifolds indexed by a partial ordered  set.   
 A topological orbifold structure on $\overline{M}_{g, n}(X, A,\omega,J)$ can be obtained (for example, cf. \cite{FO99}). Since then, there are further interests in study the smooth orbifold  structure $\overline{M}_{g, n}(X, A,\omega,J)$.  There is a similar issue for  other  compactified moduli spaces arising from geometric elliptic partial differential equations. The central part in  the study of these moduli spaces 
is the gluing analysis for the lower stratum.  The motivational  question   is how much  the gluing analysis we need in order  to obtain  a smooth structure on the moduli space.  
In this paper, we  establish a general  criterion  for an  orbifold stratified space to admit a smooth structure based on the  gluing theory.
 
 \begin{definition}  \label{def:1.1} An  $n$-dimensional
  {\em  orbifold   stratified space}   is  a  topological space $M$ which
admits a  stratification
\ba\label{eq:2.1}
 M =\bigsqcup_{\a\in \cS}  M_\a,
\na
a disjoint union of  locally closed subspaces  (called strata) indexed by a partially ordered set $(\cS, \prec)$
 such that   \begin{enumerate}
 \item The decomposition (\ref{eq:2.1})  is locally finite in the sense that each point $x\in M$ has a  neighbourhood $U_x$  such
 that $U_x\cap M_\a$ is empty except for finitely many $\a$.
\item Denote by $ \overline M_\b$ the closure of $M_\b$ in $M$. Then  $ M_\a \cap \overline M_\b \neq \emptyset \Longleftrightarrow M_\a \subset \overline  M_\b  \Longleftrightarrow  \a \prec \b.$
\item For each $\a \in \cS$, the stratum $M_\a$ has a smooth  orbifold structure given by a proper \'etale Lie groupoid
\[
 \mathsf M_\a =(M_\a^1 \rightrightarrows M_\a^0),
 \]
 that is,  $M_\a$ is  the orbit space of  $\mathsf M_\a$, also called the coarse space of $\mathsf M_\a$.
  \item The top stratum of $M$ is an $n$-dimensional smooth orbifold.
\end{enumerate}
\end{definition}
 
 In this paper, we assume that $M$ is compact and $\mc S$ is finite.
Motivated by the gluing analysis  for  the
 compactified  moduli space, such as  $\overline{M}_{g, n}(X, A, \omega, J)$, we propose a notion of linearly stratified 
 vector spaces and linearly stratified  vector bundles in Section 2. A prototype of a  linearly stratified 
 vector space is $V = \C^m$ with a stratification  given by
  \[
 V =\bigsqcup_{I\in \cS}  V^{(I)}
 \]
 with respect to the naturally ordered power set 
 \[
 \cS = 2^{\{1, 2, \cdots, m\}} = \{ I \subset \{1, 2, \cdots, m\}\}.
 \] 
 Here $V^{[I]} = \{ (t_1, t_2, \cdots, t_m)|  t_i \neq 0 \Leftrightarrow i\in I\}$,   whose closure in $V$ is a linear subspace of dimension $|I|$. One could recover the topological structure
or smooth structure on $V$ by identifying each linear  tubular neighbourhood  of $ V^{(I)} $ in $V$ with  a  normal bundle $N(V^{(I)})$  of $ V^{(I)} $  in $V$ .  See Definitions \ref{linear:strata} and  \ref{rmk:2.2} for precise definitions of  linearly stratified 
 vector spaces and linearly stratified  vector bundles. 
 This notion of linear stratifications can be generalised to an orbifold vector space  (an Euclidean vector spaces with a linear action of a finite group) and an orbifold vector bundles. 

In order to obtain a smooth orbifold structure on orbifold stratified space $M$ with respect to a finite partially ordered index set $\cS$, we first   assume that the disjoint union
\[
\mathsf M = \left( \bigsqcup_{\a\in \cS}  M_\a^1 \rightrightarrows  \bigsqcup_{\a\in \cS}  M_\a^0\right) 
\]
is a topological groupoid for simplicity.  This assumption can be dropped in actual applications. We   introduce  a   gluing datum   on
$M$, see Section 2 for details,  which  briefly consists  of 
\begin{enumerate}
\item  an orbifold linearly stratified smooth bundle  $ \Gl^\a  \longrightarrow 
 \mathsf M_\a$ with respect to the index set $\cS^\a =\{\b \in \cS | \beta\prec\alpha\}$, called a gluing bundle for each $\a \in \cS$; 
 \item and a  stratum-preserving  strict groupoid homomorphism $\phi^\a:   \Gl^\a (\epsilon)|_{\mathsf U} \to \mathsf M$, called a gluing map,  for each   open (full) subgroupoid
 $\mathsf U$ of $\mathsf M_\a$, here a metric $\mathfrak y$ is chosen on  $ \Gl^\a (\epsilon)|_{\mathsf U}$,  so the $\epsilon$-neighbourhood of   the zero section of  $\Gl^\a |_{\mathsf U}$,  is defined; moreover, the gluing map 
 is required to satisfy the following  conditions:
 \begin{enumerate}
\item $\phi^\a$ is a Morita  equivalence of   topological groupoids  from $ \Gl^\a (\epsilon)|_{\mathsf U} $ to the full-subgroupoid generated by the image of $\phi^\a$ in $\mathsf M$;
\item for each $\beta\in \cS^\a$,  the stratum-wise gluing map 
\[
\phi^\a_\b:   \Gl^\a (\epsilon)|_{\mathsf U} \to \mathsf M_\b
\]
is a Morita   equivalence of Lie groupoids  from $ \Gl^\a_\b (\epsilon)|_{\mathsf U} $ to the full-subgroupoid generated by the image of $\phi^\a$ in $\mathsf M_\b$;.
\end{enumerate}
\item a collection of  stratum-preserving smooth bundle isomorphism maps which preserve the induced  stratifications
$$
\Phi^\alpha_\beta: N({\sfGl^\alpha_\beta})\to \sfGl^\beta
$$
that covers $\phi^\alpha_\beta$ for any $\beta\in \mc S^\alpha$.  

We remark that a Morita  equivalence of proper \'etale  Lie groupoid ensures that these two Lie groupoids
are locally isomorphic, hence define the same orbifold structure on their coarse spaces. 
\end{enumerate}

Note that as $ \Gl^\a (\epsilon)|_{\mathsf U}$  is a smooth orbifold  bundle over a smooth orbifold $\mathsf U$, the gluing datum over  $\mathsf U$ defines a smooth orbifold structure on the image of the {\em coarse gluing  map}
\[
|\phi^\a| :   | \Gl^\a (\epsilon)|_{\mathsf U}|   \to  |\mathsf M| = M.
\]
A  gluing atlas  $\cF$ of $M$ is a collection of gluing data such  that  the images of all the coarse gluing  maps in $\cF$ form   an open cover of $M$. We show that a 
gluing atlas defines a canonical topological orbifold structure on $M$.  In order to achieve a smoothly compatible gluing atlas,   we further impose three conditions on a gluing atlas $\cF$ in Section 2:
\begin{enumerate}
\item[(i)] $\cF$ is closed under the induction for restriction maps and gluing maps;
\item[(ii)] $\cF$ satisfies the sewing properties;
\item[(iii)]  $\cF$ satisfies  an  inward-extendibility  condition. 
\end{enumerate}
The resulting atlas will be called  
 a {\em good orbifold gluing structure}.
 
The main theorem of  this paper is the following canonical orbifold structure on
 orbifold   stratified space $M$ equipped  with  a good orbifold gluing structure.
  \vspace{2mm}

\nin{\bf Theorem A} (Theorem \ref{orb_structure} and Remark \ref{HS:morphism})    {\em  Let $M$ be an 
   orbifold   stratified space with a good orbifold gluing structure
  $\mc F$,
   Then $M$ admits a smooth  orbifold structure such that
   each stratum is a smooth sub-orbifold.}

\vspace{2mm}

    In practice, one could  just 
 assume that $M$ is a  disjoint union of smooth orbifolds  indexed by a partially ordered set. Then the  topology on $M$ can be obtained as  a by-product of the underlying gluing data where we drop the condition of homeomorphisms of gluing maps and replace it with locally bijective maps on the set-theoretical level.  This remark is particularly  important when in applications, we often encounter  that the disjoint union
\[
\left( \bigsqcup_{\a\in \cS}  M_\a^1 \rightrightarrows  \bigsqcup_{\a\in \cS}  M_\a^0\right) 
\]
does not a priori form a  topological groupoid.  We have to replace this by a set-theoretically Morita equivalent  groupoid which  admits   a   topological or smooth orbifold structure on $M$.  
Another  technical issue is how to achieve the inward-extension property in specific applications such as
in the study of  $\overline{M}_{g, n}(X, A, \omega, J)$. This is the main motivation to develop a general principle for an orbifold stratified space, and will be addressed in a separate paper.

In the rest of the paper, as an application  of Theorem A,  we revisit the orbifold structure on the coarse  moduli space
$\overline M_{g, n}$  of stable curves of genus $g$ with $n$-marked points.   The orbifold structure on
$\overline M_{g, n}$ have been  constructed by various method in algebraic geometry \cite{ACG}  \cite{DM69}  \cite{Knu} \cite{HM}, and
in differential geometry by \cite{FO99} \cite{RoSa}.
Our construction  seems more like an  a
posteriori treatment of the known orbifold structure on $\overline M_{g, n}$. Nevertheless the novelty of our 
construction is to apply the horocycle structures at marked and nodal points to investigate the gluing
datum for $\overline M_{g, n}$. These horocycle structures satisfy the convex property which is vital to get
the inward-extendibility for a good orbifold gluing structure.

In  Section 3,  we give a preliminary review of some basics about the moduli space of stable curves. We employ the universal curves over Teichm\"uller space $\cT_{g, n}$ to get an  orbifold  structure on the  coarse   moduli space $M_{g,n}$.   Proposition \ref{DM:orbi-stra}   is a well-known result about the orbifold stratified structure on the compactified moduli space
$\overline M_{g, n}$  with respect to a partially ordered index set $\cS_{g, n}$ of weighted dual graphs of type $(g, n)$.  We include a proof of this Proposition in order to make this paper as self-contained as possible. In particular, we like to point out that each stratum $M_{[\Gamma]}$ for $[\Gamma] \in \cS_{g, n}$ admits a canonical proper \'etale groupoid $\mathsf M_{[\Gamma]}$, but the disjoint union $\sqcup_{[\Gamma] \in \cS_{g, n}} \mathsf M_{[\Gamma]}$ does not have any topological structure. 

In Section 4, we introduce a 
horocycle structure on a stable nodal Riemann surface $C$ at  a special point $p$ (a marked point or 
a nodal point). 
By a horocycle structure  of $C$ at a smooth point $p$  is a  triple $(\mathfrak y,\delta,h)$, where
$\mathfrak y$ is a metric on  the tangent space $T_pC$,  $\delta>0$ is a small constant and $h$ is a
locally defined 
 smooth map
 \[
 h:  T_p C(\delta) \longrightarrow C 
 \]
 such that $h_p(0) = p$ and the differential of $h$ at the origin is the identity operator. A hyperbolic metric on the punctured Riemann surface obtained by removing these 
 special points  from $C$ defines a canonical horocycle structure on $C$. Horocycle structures used in this paper are small perturbations of those canonical horocycle structure.  We remark  that the convexity property on the collection of  horocycle structures enables us to get a gluing atlas with the required inward-extension property. 
 In this section, we show that each stratum in $\overline M_{g, n}$, there exists a smooth family of horocycle structures in the orbifold sense. 
 
In Section 5,  we employ   the standard grafting construction to get a gluing atlas $\mc{GL}$  
 on $\overline{M}_{g,n}$ using the horocycle structures from Section 4.  
The main result of Section 5  is to  show that $\mc {GL}$ satisfies the inward-extension property.  Then by Theorem
A, $\overline M_{g,n}$ admits a smooth orbifold structure. 
 
Note that  Fukaya and Ono  in \cite{FO99}outlined  a gluing argument  to provide  $\overline M_{g, n}$ with a complex orbifold atlas.  Therefore, we don't claim any originality of this result in  this paper. 
What we have done  in some sense is to provide a full-fledge gluing theory
 for  $\overline M_{g, n}$ which is applicable to  $\overline M_{g, n}$  to get  a smooth orbifold atlas 
 on $\overline M_{g, n}$. In particular, as  we commented earlier, this  good gluing structure    will be very useful in the study of gluing theory for the compactified moduli space of
 stable maps.

\section{Smooth structures on stratified orbifolds}\label{sec:2}
 
In this section, we  will   establish  a
general  criterion for  an orbifold stratified space $M$  to admit a smooth structure.  In fact,
motivated by the gluing theory on   various moduli spaces arising from other geometric problems, we introduce the concept of gluing atlas on stratified space $M$. When $M$ admits a gluing atlas, it has  a  topological orbifold structure on  $M$ automatically.  In order to achieve the smooth compatibility we  introduce the so called "inward-extendibility" condition  (cf. Definition \ref{inward_extension}) on a gluing atlas. We call a gluing atlas with such extension property  a good gluing structure. The main theorem  in this section  is to establish Theorem A in the Introduction,  that  is, 
 $M$ admits a smooth orbifold structure if it has a good gluing structure.
We begin with the manifold stratified space first, as the arguments for  this case can be adapted easily to
orbifold stratified spaces.   

\def \mfky{\mathfrak{y}}

 \subsection{Linearly  stratified (Euclidean) spaces}\label{2.1} \ 

Let $\K$ be the ground field, like $\R$ or $\C$. Denote $\K^\times = \K \backslash \{ 0\}$.   
 Given 
  a linear $\K$-vector space  $V$ of dimension $m$  with  a fixed identification   $V\cong \K^m$ given by  coordinate functions  $(t_1, \cdots, t_m)$. on $V$ associated to a basis.

Let $\cN $ be the power set $2^{\{1,\ldots,m\}}$ with a partial order given by 
the inclusion. 
For any $I\in \cN$,   we set  
\[
V^{[I]} = \{ x\in V |  t_i (x) \neq 0 \Leftrightarrow i\in I\}. 
\]
Since $V^{[I]}\cong (\K^\times)^{|I|}$, we call it a $\K^\times$-space. let  $V^{[\emptyset]}=\{0\}$ and $V^I=\overline{V^{[I]}}$.
Then $V = V^I \times V^{I^c}$. Here $I^c$ is the compliment of $I$ in $\{1,\ldots,m\}$.  The vector  space  $V$  has a canonical  stratification 
\ba\label{can:stra}
V = \bigsqcup_{I\subset  \{1, 2, \cdots, m\}} V^{[I]}
\na
with respect to the power set  $\cN$  in the sense of Definition \ref{def:1.1}. 

The  normal bundle of $V^{[I]}$ in $V$, denoted by
$N(V^{[I]})$, can be canonically  identified with
 \[
N(V^{[I]})= V^{[I]} \times {V^{ I^c }}.
 \]
With  respect to the obvious  inclusion   $N(V ^{[I]})  \subset V$, the canonical stratification of $V$ induces a fiber-wise stratification on the vector bundle 
$N(V^{[I]}) \to  V^{[I]}$.

With these preparations, now we introduce a notion of a linear stratification on a vector space. 

\begin{definition}\label{linear:strata}
Let $V$ be a  $\K$-linear vector space  of dimension $m$   as above. We call it a {\em linearly
 stratified  space }   with respect to a partially indexed set $(\cS, \prec)$  if  $V$ admits   a  stratification  
 \ba\label{strata}
 V = \bigsqcup_{\a\in \cS} V_\a, 
 \na
with $ V_\a  =     \left(  \bigsqcup_{I \in \cJ_\a} V^{[I]} \right) $ for 
 $\mc J_\alpha\subset \mc N$,  such that  for any $\a\in \cS$, all  elements   in $\cJ_\a$  have the same  cardinality.   Denote by $(V, \cS)$ the  linearly  stratified  space $V$ with respect to $\cS$.
Given a linearly stratified  $\K$-vector space $(V, \cS)$, the group  of all invertible
linear transformation which preserves the linear stratification  is denoted by
\[
GL(V, \cS) = \{g\in GL(V) |  g(V_\a) =V_\a, \text{ for any } \a \in \cS  \}, 
\]
called the {\em stratified general linear group} of $(V, \cS)$. 
\end{definition} 
 Note that the stratification (\ref{strata})  is completely determined by the partially index  set $\cS$. The following two  conditions hold  for $(\cS, \prec)$. \begin{enumerate}
\item $2^{\{1, 2, \ldots,m \}} = \bigsqcup_{\a \in \cS} \cJ_\a$, that is, the collection of
$\{\cJ_\a\}_{\a\in \cS}$ for a partition of the power set $\cN$.
\item $\a\prec \b   \Longleftrightarrow \text{for any $I\in \cJ_\a$, there exists 
$J\in \cJ_\b$ such that $I\subset J$}. $
\end{enumerate}
In the following discussion, it might be helpful to keep one example in mind, such as a linear 
 stratification  on $V = \R^2$ with respect to $\cS =\{ \emptyset, \{\{1\}, \{2\}\}$.

Given a   linearly  stratified  space $(V, \cS)$. Let  $
N(V_\a) $ be the component-wise normal bundle of
$V_\a $ in $V$, that is,  
\[
  N(V _\a) =\bigsqcup_{I\in \mc J_\a} N(V ^{[I]}). 
  \]
As each $N(V^{[I]})$ has an induced linear stratification given by 
  \ba\label{normal:I:b}
  (N({V^{[I]}}))_\beta=N({V^{[I]}})\cap 
V_\beta=\bigsqcup_{J\in \cJ_\beta, J\supseteq I} V^{[J]}
\subset V_\beta, 
\na
with respect to $\cS^\alpha=\{\beta|\beta\succeq \alpha\}$. 
This provides a  fibre-wise linear stratification on  vector bundle 
 \[
 N({V_\a })  = \bigsqcup_{ \b \in \cS^\a}  (N(V_\a))_\beta
 \]
   over
 $ V_\a $ with respect to   $\mc S^\a$ with 
  \begin{equation}\label{normal_bundle}
 (N(V_\a))_\beta=\bigsqcup_{I\in\mc J_\a}\left(
 \bigsqcup_{J\in\mc J_\b,J\supseteq I}V^{[J]}\right).
 \end{equation}   
It is easy to see that there exists a  canonical  map 
\begin{equation}\label{bundle_embed}
\tau_{\a\b}: (N(V_\a))_\b\to V_\b 
\end{equation}  
given by the obvious inclusion  (\ref{normal:I:b}) of   each  component  in $V_\beta$  for $ I\in\mc J_\a$. 
 For simplicity, we call a map of this type  a {\em component-wise embedding}.

  \begin{lemma} \label{lemma2.2} Let $(V, \cS)$ be a linearly stratified vector space and 
  $\a\prec\b$ in $\cS$.
 Then
  the normal bundle of $(N(V_\a))_\b$ in
 $N(V_\a)$ is
\begin{equation}\label{nornormal_bundle}
N(N(V_\a)_\b) = \bigsqcup_{I\in\mc J_\a} 
\left(\bigsqcup_{J\in \cJ_\beta, J\supseteq I} 
N({V^{[J]}})\right)
\end{equation}
with a caonical component-wise embedding 
\begin{equation}\label{bunbundle_embed}
\tau^\nu_{\a\b}: N(N(V_\a)_\b)\to N(V_\b).
\end{equation}
   \end{lemma}
 \begin{proof}    Let $I\subset \mc J_\a$
and $J\in \cJ_\beta,$ such that $J\supseteq I$, 
 we have $V^{[J]} \subset N({V^{[I]}})$.
The key fact for this lemma is that the normal bundle of  $V^{[J]}$  in $N(
  {V^{[I]}})$  is same as $N({V^{[J]}})$.  As  
 $ (N({V^{[I]}}))_\beta$ consists of the disjoint union of $V^{[J]}$ for  $J\in \cJ_\beta$ with $
 J\supseteq I$,  this implies that the normal bundle of $ (N({V^{[I]}}))_\beta$ in
  $N({V^{[I]}})$ is given by the
 disjoint union of  the normal bundle $N(V^{[J]})$
 for  $J\in \cJ_\beta, J\supseteq I$,  that is,  
 $$
N( N(V^{[I]}))_\beta=\bigsqcup_{J\in\cJ_\b,J\supseteq I}
N(V^{[J]}).
 $$
 Then \eqref{nornormal_bundle} is an easy consequence of this identity. The canonical  component-wise embedding  is also obvious. 
  \end{proof}


\begin{definition}   \label{rmk:2.2}   A $\K$-vector bundle $E$  over  a smooth manifold $X$  with fiber $V$ is   linearly   stratified
if V  admits a linear  stratification with respect to a partially  ordered index 
set $(\cS, \prec)$ such that  the structure group can be reduced to $GL(V, \cS)$.
\end{definition}

A metric $\mfky$ on $V$  is compatible with the stratification in the sense that the 
coordinate function $V\cong \K^m$ is with respect to an orthonormal basis. Then we can define
\[
SO(V, \cS) =\{g\in SO(V, \mfky)|   g(V_\a) =V_\a, \text{ for any } \a \in \cS  \}. 
\]
We may equip a linearlyly stratified vector bundle $E$ with a compatible metric if the structure group can be reduced from $GL(V, \cS)$.

From Definition \ref{linear:strata},  we know that a   linearly   stratified vector bundle  $E$ has a fiberwise linear stratification with respect to $(\cS, \prec)$
\[
E=\bigsqcup_{\a\in \cS} E_\a. 
\]
Note that $E_\a$ is not a $\K$-vector bundle  as the fiber of $E_\a$ is only a $\K^\times$-vector space. 
There is   a canonical bundle   $N({E_\a}) $ over  $E_\a$, the normal bundle of the inclusion
$E_\a \subset E$,  with the  induced linear  stratification
\[
N( {E_\a})  = \bigsqcup_{\beta\in \cS^\a} (N({E_\a}))_{\b} 
\]
with respect to $\cS^\a$. By Lemma \ref{lemma2.2} we know there is a
component-wise embedding 
\begin{equation}
\tau^\nu_{\a\b}: N(N(E_\a)_\b)\to N(E_\b).
\end{equation}

    Given two   linearly   stratified $\K$-vector bundles $E$ and  $F$  over $X$ with the same index set $(\cS, \prec)$, a bundle map  $\phi:  E\to F$ is called strata-preserving if for any $\a\in \cS$,
\[
\phi(E_\a) \subseteq  F_\a.
\]
 Then one can check that the induced map on the normal bundles
  \[
  \phi:  N({E_\a} ) \to   N({ F_\a})
  \]
  is also strata-preserving.

\subsection{Gluing  principle for manifold   stratified  spaces}\label{2.2}

 An  $m$-dimensional
  manifold   stratified space  is  a topological space  $M$ which
admits a  stratification
\ba\label{eq:2.3}
M =\bigsqcup_{\a\in \cS} M_\a,
\na
a disjoint union of  locally closed smooth manifolds (called strata) indexed by a partially ordered set $(\cS, \prec)$
 such that  
  \begin{enumerate} 
  \item the dimension of its top stratum is $m$; 
 \item  the decomposition (\ref{eq:2.3})  is locally finite in the sense that each point $x\in M$ has a   neighbourhood $U_x$  such
 that $U_x\cap M_\a$ is empty except for finitely many $\a$; 
\item $M_\a \cap \overline M_\b \neq \emptyset \Longleftrightarrow M_\a \subset \overline M_\b  \Longleftrightarrow  \a \prec \b.$
\end{enumerate}
We always assume that $M$ is compact and
$\mc S$ is finite.
 It is easy to see  that 
$$
\overline{M}_\alpha\setminus M_\alpha\subseteq\bigcup_{\beta\in \mc S_\alpha} M_\beta
$$
where $\mc S_\alpha=\{\beta|\beta\prec\alpha\}$.  We
set 
$$
M^\alpha=\bigcup_{\beta\in \mc S^\alpha}M_\beta
$$ 
where $\mc S^\alpha=\{\beta|\alpha\preceq \beta\}$.
Then $M^\alpha$ is a subspace of $M$ that is
stratified by $\mc S^\alpha$.

Motivated by the gluing theory for the moduli spaces of stable maps, we impose the following  conditions on
the stratification (\ref{eq:2.3}).

\vspace{3mm}

 \nin {\bf Condition A:} ({\bf Existence of gluing bundles})   For any $\a\in \cS$, there is a linearly stratified  smooth vector bundle $\Gl^\a $  over $M_\a$ with respect  to $\cS^\a$, write
 \[
\Gl^\a =   \bigsqcup_{\beta\in \cS^\a} \Gl_\beta^\a;
\] moreover the dimension of $\Gl^\a_\b$  agrees with that of $M_\b$. This bundle is called the gluing bundle  over the strata $M_\a$. 
 \begin{remark}
 To be consistent, we allow  $\alpha$ to be the maximum element in $\mc S$, and then 
  $\Gl^\alpha= M_\alpha$, the trivial bundle with zero dimensional fiber. 
 \end{remark}

\vspace{3mm}
The linear stratification on
 $\Gl^\a$  induces   a  linearly stratification on the  normal  bundle of $ \Gl_\beta^\a$ in $ \Gl^\a$
with respect to $\cS^\b$, written as,
$$
N({ \Gl_\beta^\a})  = \bigsqcup_{\gamma\in \cS^\b}
 (N({ \Gl_\beta^\a}))_\gamma  \longrightarrow \Gl_\beta^\a
.$$ By Lemma \ref{lemma2.2}, we know that there are canonical component-wise embeddings
\ba\label{db:include}
\tau_{\b\gamma}: (N({ \Gl_\beta^\a}))_\gamma
 \to { \Gl_\gamma^\a}; \;\;\;
\tau^\nu_{\b\gamma}: N((N({ \Gl_\beta^\a}))_\gamma) 
 \to  N({ \Gl_\gamma^\a}).
\na

For a linearly stratified  vector  bundle $E$ with a compatible metric $\mfky$, we denote its $\delta$-ball bundle by
$E(\delta)$ for any $\delta>0$.  

\begin{definition}[Gluing datum]
Let $U$ be any open  subset of  $M_\alpha$. By a gluing datum over $U$ we mean
a metric $\mfky^\alpha$ on $\Gl^\alpha|_U$, a stratified map (called a gluing map)
$$
\phi^\alpha: \Gl^\alpha(\epsilon)|_U\to M^\alpha
$$
for some constant $\epsilon>0$ and 
a collection of  stratum-preserving smooth bundle isomorphisms 
$$
\Phi^\a_\b: N(\Gl^\a_\b)\to \Gl^\b
$$
such that 
\begin{enumerate}
\item the image of $\phi^\alpha$ is open and the map is  a homeomorphism  onto its image;
\item $\phi^\alpha$ is a stratified smooth map with respect to the induced  stratifications,  i.e, for any 
$\beta\in \mc S^\alpha$, $\phi^\alpha$ maps $\Gl^\alpha_\beta(\epsilon)|_U$ to $M_\beta$, we denote this map by 
 $$
 \phi^\alpha_\beta: \Gl^\alpha_\beta(\epsilon)|_U\to M_\beta,
 $$
 then $\phi^\alpha_\beta$ is a  diffeomorphism  onto its image;
 \item  the bundle isomorphism $\Phi^\a_\b$ covers $\phi^\alpha_\beta$ in the sense that the diagram 
 \[
 \xymatrix{
 N(\Gl^\a_\b)\ar[d]\ar[r]^{\Phi^\a_\b} & \Gl^\b\ar[d]
 \\
  \Gl^\alpha_\beta(\epsilon)|_U\ar[r]^{\phi^\a_\b} & M_\beta
  }
 \]
 commutes.
\end{enumerate}
We simply denote this gluing datum over $U$ by 
$$(U,\mfky^\alpha,\epsilon, \phi^\alpha, \Phi^\a:=\{\Phi^\alpha_\beta\}_{\beta\in \mc S^\alpha})$$ a gluing datum over $U$. 
\end{definition}


\begin{definition}\label{gluing_structure}  Let  ${\mc{GL}}(M) $ be the collection of all gluing data.
Let $\mc F$ be  a subset  of  ${\mc{GL}}(M)$. We call $\mc F$ a {\em gluing atlas}  of $M$
if   the image of the gluing maps in $\cF$ forms an open cover of $M$. 
\end{definition}

\begin{theorem}\label{mfd_structure}
Suppose that a manifold  stratified space $M$ has a gluing atlas
 $\mc F$. Then $M$ admits a canonical topological 
manifold structure defined  by $\mc F$.  \end{theorem}
\begin{proof}
Note that given a gluing datum 
$$
(U,\mfky, \epsilon, \phi^\a,\Phi^\a)
$$
the map 
$$
\phi^\a: \Gl^\a(\epsilon)|_U\to M^\a
$$
gives  a manifold  topological  structure on the image of $\phi^\a$, since
$\Gl^\a(\epsilon)|_U$ is a (smooth) manifold. 
For any point $x_\alpha\in M_\alpha$ there exists a small neighborhood $U_\alpha$ that is
proper in $M_\alpha$,  an induced  gluing datum $(U_\alpha,\mfky,\epsilon,
\phi^\alpha,\Phi^\alpha)$  from a gluing datum in $\cF$ by restriction to $U_\a$. Moreover, $U_\a$ can be chosen 
such that $\Gl^\a (\epsilon)|_{U_\a}$ is trivial over $U_\a$ and  is  homeomorphic to   $\R^m$.  Then 
$$
\phi^\alpha: \Gl^\alpha(\epsilon)|_{U_\a} \to M^\alpha\subset M, 
$$
defines a coordinate chart for a neighborhood $U_{x_\a} = \phi^\alpha ( \Gl^\alpha(\epsilon)|_{U_a}) $ 
 of $x_\alpha$ in $M$, denoted by
 \[
 \psi_{x_\a} = (\phi^\alpha)^{-1}:   U_{x_\a}  \longrightarrow  \Gl^\alpha(\epsilon)|_{U_\a}  \cong  \R^m.
  \]
  So   locally we have 
a topological manifold structure on $M$. The transition functions 
\[
\psi_{x_\a}    \circ  \psi_{y_\b} ^{-1}:   \psi_{y_\b} (  U_{x_\a} \cap U_{y_\b} ) 
\longrightarrow  \psi_{x_\a} (  U_{x_\a} \cap U_{y_\b} )
\]
on overlaps  $U_{x_\a} \cap U_{y_\b}$ for  any $y\in M_\b$ are homeomorphisms of open subsets of $\R^m$,  as each  gluing map 
  is a homeomorphism onto its image.  Therefore, $M$ has a canonical topological 
manifold structure.  
\end{proof}
The next condition is motivated by standard gluing theory for moduli spaces.
\vskip 0.1in
\noindent
{\bf Condition B: (Existence of gluing data).} For any $\a\in\mc S$ and
any proper open subset $U\subset M_\a$ there exists a gluing datum over $U$.
\vskip 0.1in
It is clear that Condition B implies trivially  the existence of gluing atlas. 

Now we  come to the smooth structure on $M$ which would follow if we have  a $C^\infty$-compatible gluing atlas 
in the sense that all the transition functions $\psi_{x_\a}    \circ  \psi_{y_\b} ^{-1}$ in the proof of  Proposition
\ref{mfd_structure} are diffeomorphsims of subsets of $\R^m$.  We begin with the following  observations for ${\mc{GL}}(M) $. 
\begin{enumerate}
\item[(i)]  ({\em Induction for restriction maps})  Suppose that $(U,\mfky^\alpha, \epsilon, \phi^\alpha, {\Phi^\alpha})$ is a gluing datum, then
for any open subset $U'\subseteq U$ and $0<\epsilon'\leq \epsilon$, by taking the restriction of maps 
we have an obvious induced gluing datum
$(U',\mfky^\alpha, \epsilon',\phi^\alpha,{\Phi^\alpha})$. Clearly, two coordinate charts are $C^\infty$-compatible if
one of their associated gluing datum is obtained from the restriction of the other gluing datum.

\item[(ii)]  ({\em Induction for   gluing maps})  Suppose that $\alpha\prec\beta$ be a pair in $\mc S$, 
let 
$$
(U,\mfky,\epsilon, \phi^\a,\Phi^\a)
$$
be a gluing datum over $U\subset M_\a$. Fix
any $\b\in \mc S^\a$, and let $D(\a,\b)$ and
$R(\a,\b)$ be the domain and image of $\phi^\a_\b$. Then
for any {\em proper} 
open subset $U^\b\subset R(\a,\b)$ we have  gluing
data over $U^\b$ defined as the following. Note that 
$$
\Phi^\a_\b: N(\Gl^\a_\b)\to \Gl^\b
$$
is a bundle isomorphism covering  
$$
\phi^\a_\b: D(\a,\b)\to R(\a,\b).
$$
The metric $\mfky$ on $\Gl^\a $ induces
a metric on $N(\Gl^\a_\b)$. Under the bundle isomorphism $\Phi_\b^a$, we get a 
metric on  $\Gl^\b|_{R(\a,\b)}$  denoted by $\mfky^\b$. Define $$
\psi:\Gl^\b|_{R(\a,\b)}
\xrightarrow{(\Phi^\a_\b)^{-1}}N(\Gl^\a_\b)|_{D(\a,\b)}
\xrightarrow{\tau}\Gl^\a
$$
to be the composition of $(\Phi^\a_\b)^{-1}$ and the canonical component-wise  embedding $\tau$. 
We can  choose $\epsilon'$ small enough   such that 
$$\psi:\Gl^\beta(\epsilon')|_{U^\b}\to \Gl^\a(\epsilon)|_U$$
is in fact  an embedding.
Set  
\begin{equation}\label{induced_maps}
\phi^\b=\phi^\a\circ\psi,\;\;\;
\Phi^\b_\gamma=\Phi^\a_\gamma\circ\psi.
\end{equation}
Then 
$ 
(U^\b, \mfky^\b, \epsilon', \phi^\b,\Phi^\b)
$ 
is a gluing datum over $U^\b$.  It is also easy to check that two  coordinate charts are $C^\infty$-compatible if
one of their associated gluing datum is obtained  by the induction for the gluing map in the other gluing datum.
This follows from the fact that, under the canonical component-wise embeddings $\tau_{\b\c}$ and $\tau_{\b\c}^\nu$ in (\ref{db:include}), we have
\begin{equation}\label{strongcondition}
\phi^\a=\phi^\b\circ \Phi^\a_\b,\;\;\;
\Phi^\a_\gamma=\Phi^\b_\gamma\circ \Phi^\a_\b.
\end{equation}
for $\a\prec\b\prec\gamma$ in $\cS$. 

\item[(iii)] ({\em Sewing property})  Let 
$$
 (U_k, \mfky_k,\epsilon_k, \phi^\alpha_k, \Phi^\alpha_{k}), k=1,2
 $$
be two gluing data, where $U_1$ and $U_2$
are both open subsets of $M_\a$. We say that they {\em coincide} if, on the intersection domain $V=U_1\cap U_2$,
\begin{itemize}
\item $\mfky_1=\mfky_2$ on $\Gl^\alpha|_V$; and
\item  
$
\phi^\alpha_1=\phi^\a_2, \Phi^\a_{1}=\Phi^\a_{2}$
on common domains. 
\end{itemize}
Given such a pair, it is obvious that we
can  sew  them together to  get a new gluing datum
over $U=U_1\cup U_2$:
\begin{itemize}
\item  $\mfky_{1}$ and $\mfky_2$ together yield a metric 
$\mfky$ on $\Gl^\a$ over $U$, 
\item take 
$
\epsilon<\min(\epsilon_1,\epsilon_2),$
\item  the gluing map
$\phi^\a$ and  the bundle isomorphisms  $\{ \Phi^\a_{\b} \} $ are defined in an obvious way.  
\end{itemize}
Such a gluing datum over $U= U_1\cup U_2$  is called a sewed gluing 
 datum of the  two coincided data. 
\end{enumerate}
 
Motivated by this sewing property, we introduce  the inward-extendibility condition for gluing data over boundary-type
open subsets.

  \begin{definition}
For a stratum $M_\a$, we say that an open subset $U\subset M_\a$ is of boundary-type if $M_\a\setminus U$ is closed in 
$\bar M_\a$.
\end{definition}

\begin{lemma}
Denote $\bar M_\a\setminus M_\a$ by $\partial M_\a$.
Suppose $U$ is an open subset of $M_\a$. Then  $U$ is of boundary-type if and only if
$U\cup \partial M$ is open in $\bar M_\a$.
\end{lemma}
\begin{proof}
Note that 
$
M_\a\setminus U=\bar M_\a\setminus (U\cup\partial M).
$ Then the lemma is a consequence of this fact.
\end{proof}
\begin{definition}\label{inward_extension}
Let $A:=(U,\mfky, \epsilon, \phi^\a,\Phi^\a)$ be a gluing datum 
over a boundary-type open subset $U\subset M_\a$. 
We say that  a gluing datum over $M_\a$
 $$
\tilde A:=(M_\a,\tilde\mfky,\tilde\epsilon,\tilde\phi^\a,\tilde\Phi^\a)
$$ 
is an inward-extension of $A$ if there exists
a  boundary-type open subset $U'\subset U$ such that 
$A$  agrees with $\tilde A$ over $U'$.
\end{definition}

\begin{definition}\label{good:gluing_atlas}
A gluing atlas  $\mc F$ is called  a {\em good gluing structure}  of $M$
if $\mc F$ satisfies the following conditions:
\begin{enumerate}
\item[(i)]  $\mc F$ is closed under the induction for  restriction maps and closed under the induction for  gluing maps;
\item[(ii)]  $\mc F$ satisfies the sewing property in the sense that if $\cF$ has  a  pair  gluing data over open subsets  of a stratum of $M$ which coincide over their intersection, then 
their   sewed datum  is also in $\mc F$. 
\item[(iii)] any boundary-type  gluing datum in $\cF$ has an inward-extension in $\cF$. 
\end{enumerate}
\end{definition}
\vskip 0.1in

 \nin {\bf Condition C:} ({\bf Existence of a good gluing structure}) 
 There exists a good  gluing structure  $\cF$ in $\mc{GL}(M)$ for the manifold stratified space $M$. 
 
  \vspace{3mm}
 
  The following lemma implies that the $C^\infty$-compatibility condition can be checked by applying the induction for
  gluing maps.  The proof follows directly from the $C^\infty$-compatibility under the  induction for
  gluing maps.

\begin{lemma}\label{lem_2.6}
Two coodinate charts associated to gluing data
\begin{equation}\label{eqn_2.12}
(U_{\alpha_k},\mfky_k^{\alpha_k},
\epsilon^{\alpha_k},\phi^{\alpha_k}, \Phi^{\alpha_k})
\end{equation}
over $U_{\alpha_k}$ for $k=1,2$ are $C^\infty$-compatible if
for any $\beta\in \mc S^{\alpha_1}\cap \mc S^{\alpha_2}$ 
 their  induced gluing data  obtained from the induction for gluing maps agrees over the  common domain. 
\end{lemma}
 
\begin{definition}[Gluing-compatibility]\label{def_2.7}
Two gluing data
$$
(U_{\alpha_k},\mfky^{\a_k},\epsilon^{\alpha_k},\phi^{\alpha_k}, \Phi^{\alpha_k})
$$
over $U_{\alpha_k}$ for $k=1,2$ are said to be gluing-compatible if
for any $\beta\in \mc S^{\alpha_1}\cap \mc S^{\alpha_2}$,
any their induced gluing data on the common domain in  $M_\b$ by  gluing maps 
agree.   
\end{definition}

We  finally come to the conclusion that a gluing atlas $\cF$  (Cf. Definition \ref{gluing_structure}) for $M$ defines a canonical smooth structure on $M$ if
any pair in $\cF$ is gluing-compatible.

\begin{theorem}\label{smooth_structure}
Suppose that $M$ admits a good gluing structure
$\mc F$.
Then there exists a particular gluing atlas 
$$
\mc G:=\{A_\alpha=(M_\alpha, \mfky^\alpha,\epsilon^\alpha,\phi^\alpha,\{\Phi^\alpha_\beta\}_{\beta\in \mc S^\alpha})\}_{\alpha\in \mc S}\subset \mc F
$$
such that any pair in $\mc G$  are gluing-compatible.   
 \end{theorem}
 
\begin{proof} 
We first introduce subsets $\mc S_1,\mc S_2,\ldots $ of
$\mc S$  inductively by letting $\mc S_n$ consist of all smallest elements in 
$$
\mc S\setminus \bigcup_{k<n} \mc S_k.
$$
Here we assume that $\mc S_0=\emptyset$.
Write \begin{equation}
\mc S=\mc S_1\cup\cdots\cup \mc S_K.
\end{equation}

We will construct a collection gluing datum 
$$
 \{ A^{(k)}_\alpha=(M_\a, \mfky^\a,
 \epsilon^\a_{(k)}, \phi^\a,\Phi^\a) |  \a\in \bigcup_{i=1}^k \mc S_i\}
 $$ 
by applying an  induction argument  to $k$
satisfying  the following properties:
\begin{enumerate}
\item[(I)] any pair $A^{(k)}_\a$ and $A^{(k)}_\b$ are
gluing-compatible;
\item[(II)] 
 for any incomparable 
indices  $\alpha$ and $\beta$
\begin{equation}\label{eqn_seperable}
R^{(k)}(\phi^\alpha)\cap R^{(k)}(\phi^\beta)\subset \bigcup_{\gamma\in \mc S_\alpha\cap\mc S_\beta} R(\phi^\gamma).
\end{equation}
Here $R^{(k)}(\phi^\alpha)$ is the image of  the gluing map $\phi^\alpha$ in $A^{(k)}_\a$ and  $\mc S_\alpha$
is defined to be $ \{\gamma|\gamma\preceq \alpha\}$. 
 Though the domains
 of $\phi^\a$
 in $A^{(k)}_\a$ depend on $k$ (due to the changes of $\epsilon^\a_{(k)}$), the maps are same on
 common domains. Hence we simply denote it by $\phi^\a$
 without indicating  $k$.
\end{enumerate}
 \vskip 0.1in
\noindent {\em \underline{Step 1}  ($k=1$).}  Note that any $\alpha\in \mc S_1$, $M_\alpha$ is compact and 
any two  of $\{M_\a | \a \in \cS_1\}$  are disjoint. Choose an arbitrary gluing datum 
$$
{A}^{(1)}_\alpha:=(M_\alpha,\mfky^\a,
 \epsilon^\alpha_{(1)},
  \phi^\alpha, \Phi^\alpha)
$$  
with
 small enough $ \epsilon^\alpha_{(1)}$ such that 
 $ 
R(\phi^\a)\cap R(\phi^\b)=\emptyset.
$ 
Hence, $\{{A}^{(1)}_\alpha | \a \in \cS_1 \}$  trivially  satisfies the properties (I) and (II). 

\vskip 0.1in
\noindent {\em \underline{Step 2}  ($k=2$).} For any 
$\alpha\in \mc S_2$,  by the induction for the gluing maps, the gluing data 
$\{ A^{(1)}_\gamma | \c \in \cS_1\}$
 induce a gluing datum over 
$$
U_\alpha=\bigcup_{\gamma\prec\alpha} R(\phi^\gamma_\alpha) \subset M_\a.
$$
 We denote it by 
$$
 B_\alpha= (U_\alpha,\mfky, \epsilon, \tilde\phi^\alpha,\tilde\Phi^\alpha)
$$
for some small $\epsilon$.
Clearly, $U_\alpha$ is a boundary-type open subset of $M_\alpha$. Then by the assumption that $\mc F$ is a  good
gluing structure,  $B_\alpha$ has an inward-extension, denoted by 
$$
A^{(2)}_\a=(M_\a, \mfky^\a, \epsilon^\a_{(2)}, \phi^\a,\Phi^\a)
$$
such that it coincides with $B_\alpha$ over some
boundary-type open subset $V_\a$ of $U_\alpha$.

 Next we need do some modifications on existing gluing data $\{A_\a^{(2)} | \a\in  \cS_1 \cup \cS_2\}$ as follows. 
\begin{enumerate}
\item For any $\gamma\in \mc S_1$ we replace $\epsilon^\gamma_{(1)}$ by smaller $\epsilon^\gamma_{(2)}$ such that 
the image of the gluing mp $\phi^\gamma_\alpha$ is a open subset of  $V_\a$ for $\a \in \cS_1 \cup  \cS_2$; 
\item we may  downsize $\epsilon^\a_{(2)}$, for example, 
$ {\epsilon}^\alpha_{(2)}
\leq \frac{1}{2}
\min_{\gamma\in\mc S_1}\epsilon^\gamma_{(1)},  $ 
 such that \eqref{eqn_seperable} holds for any incomparable pair in $\cS_2$. 
 \end{enumerate}
We now verify that any two gluing data  ${A}^{(2)}_\alpha$ and ${ A}^{(2)}_\beta$ are gluing-compatible:
\begin{enumerate}
\item if $\alpha$ and $\beta$ are comparable, say $\alpha\prec\beta$, then the induced gluing datum
 from ${A}^{(2)}_\alpha$  
 is over a subset  of $V_\b$ and it coincides with
${ A}^{(2)}_\beta$ by the construction;
\item if $\alpha$ and $\beta$ are incomparable and $ \mc S_\alpha\cap \mc S_\beta=\emptyset$, then two gluing data are trivially compatible due to \eqref{eqn_seperable};
\item if $\alpha$ and $\beta$ are incomparable and $\mc S_\alpha\cap \mc S_\beta\not=\emptyset$, then the image of $\phi^\alpha$ and $\phi^\beta$ are covered
by images of  $\phi^\gamma$ of   $\gamma\in\mc S_{\alpha,\beta}$, but 
both $ A^{(2)}_\alpha$ and $A^{(2)}_\beta$
 are gluing-compatible with $A^{(2)}_\gamma$, this implies that they are gluing-compatible with each other.  
\end{enumerate}

\vskip 0.1in
\noindent {\em \underline{Step 3}  (general case).}  Now suppose that $A^{(k)}_\alpha$ for all $\alpha$ in $\mc S_1,\ldots,
\mc S_k$ are constructed. We proceed to  construct $ A^{(n+1)}_\alpha$ for all $\alpha\in \mc S_{k+1}$ by repeating the  same   construction as in Step 2. That is, for any $\a\in \cS_{k+1}$, the gluing data  $\{ A^{(1)}_\gamma | \c \prec 
\a\}$ defines a boundary-type  gluing over a subset in $M_\a$. The  inward-extension
condition supplies us with a gluing datum 
\[
A^{(k+1)}_\a=(M_\a, \mfky^\a, \epsilon^\a_{(2)}, \phi^\a,\Phi^\a).
\]
 Then as in Step 2,  we modify
$\epsilon_{(k)}^\a$ to a suitable $\epsilon^\a_{(k+1)}$
for  $\alpha$  in $\mc S_1,\ldots,\mc S_k$ and get an updated
$\{A^{(k+1)}_\a\}$  satisfying both (I) and (II). 

Since $\mc S$ is finite,  such a procedure will terminate in finite steps. Hence, we get a required 
gluing atlas as in the Theorem. \end{proof}

\begin{remark}\label{cor_metric}
Let $\mc G$ be the collection of gluing data given in Theorem \ref{smooth_structure}.
Then all the bundle isomorphisms $\Phi^\a_\b$ in $\mc G$ are isometric.
\end{remark}

\def \sfG{\mathsf G}
\def \sfE{\mathsf E}

\subsection{Gluing theorem for  orbifold  stratified  spaces}\label{2.4} \ 

In this subsection, we generalise the results in previous subsections to orbifold stratified spaces. We  employ the language of proper \'etale groupoids to describe  topological and smooth orbifold  following  the definition  of  proper \'etale groupoids as  in  \cite[Definition 2.6]{CLW}.
 For readers' convenience, we  recall the definitions of  Lie groupoid, proper \'etale groupoids  and   vector bundles over   Lie  groupoids from \cite{CLW}.  Topological groupoids and vector bundles over topological groupoids can be defined in a similar way.

\begin{definition}  \label{Lie-gpoid:def} ({\em Lie groupoids and proper \'etale groupoids})
 A Lie groupoid
 $\sfG = (G^0,  G^1)$  consists of two smooth manifolds $G^0$ and $G^1$, together with five smooth maps $(s, t, m, u, i)$ satisfying the following properties.
  \begin{enumerate}
\item  The source map  and the target map $s, t: G^1 \to  G^0$ are submersions.
\item The composition map
\[
m:  G^{[2]}: =\{(g_1, g_2) \in  G^1 \times  G^1: t(g_1) = s(g_2)\} \longrightarrow G^1
\]
written as $m(g_1, g_2) = g_1\circ  g_2$ for composable elements $g_1$ and $g_2$,
satisfies the obvious associative property.
\item The unit map $u: G^0 \to G^1$ is a two-sided unit for the composition.
\item The inverse map $i: G^1 \to G^1$, $i(g) = g^{-1}$,  is a two-sided inverse for the composition.
\end{enumerate}
In this paper,   a groupoid $\sfG$   will be denoted by  $\sfG = (G^1 \rightrightarrows G^0)$ where $G^0$ will be called the space of objects or units, and $G^1$ will be called the space of arrows.
A Lie groupoid $\sfG$   is {\em  proper } if $(s, t): G^1 \to G^0 \times G^0$ is proper, and  is called {\em   \'etale } if $s$ and $t$ are local diffeomorphisms.   Given a proper \'etale groupoid $ (G^1 \rightrightarrows G^0)$, for any $x\in G^0$, 
\[
G_x = (s, t)^{-1}(x, x) = s^{-1}(x)\cap t^{-1}(x)
\]
 is a finite group, called the {\em isotropy group}  at $x$.

\end{definition}

\begin{remark}\label{opensubset}  Let $\sfG =  (G^1 \rightrightarrows G^0)$ be a proper \'etale Lie groupoid. 
\begin{enumerate}
\item We remark that 
$G^1$ defines an equivalence relation on $G^0$: that is, any two points in $G^0$
are equivalent if they are the source and target of an arrow in $G^1$. The quotient space
$G^0/\sim$ is denoted by $|\sfG|$ and is called the coarse space of $\mathsf G$. Let 
$\pi: G^0\to |\sfG|$ be the projection map. There is a canonical orbifold structure on $|\sfG|$
defined by $\mathsf G$. 

\item In this paper, when we say that $\mathsf U=(U^1\rightrightarrows U^0)$  is an open {\em full-subgroupoid }
of $\mathsf G$ if it is of the form 
$$
U^0=\pi^{-1}(V), \;\;\; U^1= s^{-1}(U^0),
$$
for an open subset of $V$ of  $|\mathsf G|$. 
For example, given any open subset $U\subset G^0$ we can associate it an open full-subgroupoid $\mathsf U$
by setting 
$$
U^0=\pi^{-1}(V), \;\;\; U^1= s^{-1}(U^0),
$$
where $V=\pi(U)$.
\item Let $x\in G^0$, there is a $G_x$-invariant open neighbourhood of $x$ in $G^0$ such that the full-subgroupoid associated to $U_x$  is Morita equivalent to the action groupoid
\[
U_x \rtimes G_x \rightrightarrows U_x.
\]
This latter  action groupoid is called a local model of $\sfG$ at $x\in G^0$. 
Recall that for  proper \'etale Lie groupoids, a Morita equivalence means that their  coarse spaces are homeomorphic and  their local models are isomorphic. 
\end{enumerate}
\end{remark}

\begin{proposition} \label{bundle:gpoid}
Given a Lie groupoid  $\sfG=(G^1 \rightrightarrows G^0)$,  a Lie groupoid
$\sfE=( E^1  \rightrightarrows E^0)$ is a vector bundle over $\sfG$ if and only if there is
 a strict Lie groupoid morphism $\pi:  ( E^1  \rightrightarrows E^0) \to (G^1 \rightrightarrows G^0)$ given by
 the commutative diagram
\ba\label{pull-back:gpoid}
\xymatrix{
E^1 \ar@<.5ex>[d]\ar@<-.5ex>[d] \ar[r]^{\pi_1} & G^1\ar@<.5ex>[d]\ar@<-.5ex>[d]\\
E^0 \ar[r]_{\pi_0}& G^0 }
\na
in the category of Lie groupoids with strict morphisms, such that
\begin{enumerate}
\item  the diagram (\ref{pull-back:gpoid}) is a pull-back groupoid diagram,
\item both $\pi_1: E^1\to G^1$ and $ \pi_0: E^0\to G^0$
are vector bundles.
\item  the pull-back arrows
\[
\{(v_x, \gamma, v_y) | \gamma \in G^1, (v_x, v_y) \in E_{s(\gamma)} \times E_{t(\gamma)} \}
\]
define  a linear isomorphism  $\xi (\gamma):  E_{s(\gamma)} \to E_{t(\gamma)} $ sending $v_x$ to $x_y$.
\end{enumerate}
\end{proposition}

We say that  $( E^1  \rightrightarrows E^0)$ is linearly stratified
if $E^1$ and $E^0$ are linearly stratified such that the commutative diagram  (\ref{pull-back:gpoid})
preserves the linear stratifications, and moreover,  the pullback  arrows in Proposition
\ref{bundle:gpoid} is a stratum preserving isomorphism, or simply, the arrows in
$E^1$ preserve the linear stratification on $E^0$.

 Now we consider an  orbifold stratified space  as Definition   \ref{def:1.1} 
 $$M=\bigsqcup_{\a\in \cS} |\sfM_\a|,$$
 where $ 
 \sfM_\alpha=(M_\a^1\rightrightarrows M_\a^0) $ is a proper \'etale Lie groupoid.
Denote
 $$
 \sfM=(M^1\rightrightarrows M^0)=\left(\bigsqcup_{\a\in \cS}    M^1_\a\rightrightarrows
  \bigsqcup_{\a\in \cS} M^0_\a \right). 
 $$
We {\bf assume} that $\sfM$ is a proper \'etale topological groupoid and  both
$M^0 $ and $M^1$ are manifold  stratified spaces with respect to $\cS$. We can adapt all the arguments for manifold stratified spaces to orbifold  stratified spaces.

 \vspace{3mm}

 \nin {\bf Condition A'}  ({\bf Existence of orbifold gluing bundles})   For any $\a\in \cS$, there is a linearly stratified  smooth  orbifold vector bundle $$\sfGl^\a
 =(\Gl^{1,\a}\rightrightarrows\Gl^{0,\a})\to\sfM_\a
 =( M^1_\a\rightrightarrows M^0_\a)$$ with respect  to $\cS^\a=\{\beta|\alpha\preceq \beta\}$, such that for
  $i=1,2$,  $\Gl^{i,\alpha} \to M^i_\alpha $ is a gluing bundle  for the manifold stratified space  $ M^i_\alpha$.   
 This bundle is called the orbifold gluing bundle  over the strata $\sfM_\a$.

\vspace{3mm}

We can equip   the gluing bundle $\sfGl^\a$  with  a compatible smooth   metric $\mathfrak y^\a$ so that
for any $\epsilon>0$,  the open $\epsilon$-ball  bundle of $\sfGl^\alpha$
 with the induced stratification.  Set  $\mc S_\alpha=\{\beta|\beta\prec\alpha\}$. Let 
 $$
\sfM^\alpha=\bigcup_{\beta\in \mc S^\alpha}\sfM_\beta
$$ 
where $\mc S^\alpha=\{\beta|\alpha\preceq \beta\}$.

\begin{definition}\label{orb_gluing_datum}
Let $\mathsf U=(U^1\rightrightarrows U^0)$ be any open full-subgroupoid of  $\sfM_\alpha$. 
A  gluing datum over $\mathsf U$  consists of  a metric $\mfky=(\mfky^0,\mfky^1)$ on $\sfGl^\a|_{\mathsf U}$ and a gluing map 
$
\phi^\alpha: \sfGl^\alpha(\epsilon)|_{\mathsf U}\to \sfM^\alpha
$ given  by a strict morphism of topological groups
$$
(\phi^{1,\alpha},\phi^{0,\alpha}): (\Gl^{1,\alpha}(\epsilon)|_{U^1}\rightrightarrows
 \Gl^{0,\alpha}(\epsilon)|_{U^0})\to ( M^{1,\alpha}\rightrightarrows M^{0,\alpha})
$$ 
for some constant $\epsilon>0$ such that $\phi^{i,\alpha}$, for $ i=0,1$, are gluing maps for 
$M^i$, namely,
\begin{enumerate}
\item the image of $\phi^{i,\alpha}$ is open and the map $\phi^{i,\alpha}$ is
a homemorphism   onto its image  in the sense of topological groupoids;
\item the map $\phi^\alpha$ is a stratified smooth map with respect to the stratification, i.e, for any 
$\beta\in \mc S^\alpha$, $\phi^\alpha$ maps $\sfGl^\alpha_\beta(\epsilon)|_{\mathsf U}$ to $\sfM_\beta$, we denote this map by 
 $$
 \phi^\alpha_\beta: \sfGl^\alpha_\beta(\epsilon)|_{\mathsf U}\to \sfM_\beta;
 $$
 then $\phi^\alpha_\beta$ is an isomorphism  onto its image    in the sense of Lie groupoids (cf. Remark \ref{gp_diff});
\end{enumerate}
and a collection of  stratum-preserving smooth bundle isomorphism maps which preserve the induced  stratifications
$$
\Phi^\alpha_\beta: N({\sfGl^\alpha_\beta})\to \sfGl^\beta
$$
that covers $\phi^\alpha_\beta$ for any $\beta\in \mc S^\alpha$  in the sense that the diagram 
 \[
 \xymatrix{
 N(\sfGl^\a_\b)\ar[d]\ar[r]^{\Phi^\a_\b} & \sfGl^\b\ar[d]
 \\
  \sfGl^\alpha_\beta(\epsilon)|_{\mathsf U}\ar[r]^{\phi^\a_\b} & \sfM_\beta
  }
 \]
 commutes in the category of Lie groupoids and strict morphisms.
We  denote this  gluing datum over $\mathsf U$ by 
\[
(\mathsf U,\rho,\epsilon, \phi^\alpha, \{\Phi^\alpha_\beta\}_{\beta\in \mc S^\alpha}).
\]
 \end{definition}
\begin{remark}\label{gp_diff}
Let $R^i(\alpha,\beta)$ be the image of $\phi^{i,\a}_\b, i=0,1$
and $|R(\alpha,\beta)|$ be the image of $|\phi^\a_\b|$ (the coarse map of $\phi^a_\b$).
Then
\[
\mathsf R(\a,\beta) = (R^1(\a,\b)\rightrightarrows R^0(\a,\b))
\]
is the 
  full subgroupoid associated to an open subset $|R(\a,\b)|$ of $M_\b$.  Then  $\phi^\alpha_\beta$, being an isomorphism  onto its image    in the sense of Lie groupoids,   is a strict Lie  groupoid isomorphism 
\begin{equation}\label{glue-up-morita}
 \sfGl^\alpha_\beta(\epsilon)|_{\mathsf U}\ \cong 
\mathsf R(\a,\b).
\end{equation}
This is equivalent to say that the maps $\phi^{\a, 0}_\b$ and $\phi^{\a, 1}_\b$ are diffeomorphisms onto its images
in $M_\b^0$ and $M_\b^1$ respectively. 
\end{remark}

Parallel to the manifold stratified case, we assume Condition B' and  Condition C'. 
\vskip 0.1in
\noindent
{\bf Condition B': (existence of gluing data)} For any
$\a\in \mc S$ and any proper open subset $\mathsf U\subset \sfM_\a$
there exists a gluing datum over $\mathsf U$.
\vskip 0.1in
One
can also define  a {\em good  orbifold gluing structure}  $\cF$ for an orbifold stratified space as follows. 
\begin{definition}\label{good-orbi:gluing_atlas}
A {\em good  orbifold gluing structure}  $\cF$ for an orbifold stratified space $M$
is a collection of gluing data  satisfying  the following conditions:
\begin{enumerate}
\item[(i)]  the image of coarse gluing maps  associated  to $\mc F$ forms an open cover of $M$;
\item[(ii)]   $\mc F$ is closed under the induction for  restriction maps and closed under the induction for  gluing maps;
\item[(iii)]  $\mc F$ satisfies the sewing property in the sense that if $\cF$ has  a  pair  gluing data over open fullgroupoids  of a stratum of $\sfM$ which coincide over their intersection, then 
their   sewed datum  is also in $\mc F$. 
\item[(iv)] any boundary-type  gluing datum in $\cF$ has an inward-extension in $\cF$. 
\end{enumerate}
\end{definition}
\vskip 0.1in

 \nin {\bf Condition C'}   ({\bf Existence of  good orbifold gluing structure})  There exists a  good orbifold gluing structure
 $\cF$ for  the orbifold stratified space $M$. 
 \vskip 0.1in
 
We remark that a good orbifold gluing structure for  the  orbifold stratified space $M$   provides good  gluing structures for manifod stratified spaces $M^0$ and $M^1$.  The following theorem implies Theorem A in the Introduction.   The proof  is to apply  the same arguments in the proof of Theorem \ref{smooth_structure} to good gluing structures to $M^0$ and $M^1$ such that $ M^1\rightrightarrows M^0$ is a proper \'etale groupoid.  
 
\begin{theorem}\label{orb_structure}
Suppose that the orbifold
stratified space $ M$ has a  good orbifold 
gluing structure $\mc F$,
  then there exists a  particular gluing data 
  $$
\cG = \{ \mc A_\alpha=(\sfM_\alpha, \mfky^\a,
\epsilon_\alpha,\phi^\alpha,\Phi^\alpha)| \a \in \cS\} \subset \cF
$$
such that any pair in $\cG$  are gluing-compatible, hence, $C^\infty$-compatible. 
\end{theorem}

\begin{remark}\label{HS:morphism}
Using  the language of proper \'etale Lie (or topological) groupoids  to describe smooth (or 
topological) orbifolds, the correct notion of  morphisms between two groupoids should be generalised morphisms
 in the sense of \cite{HS} instead of
strict morphisms, and generalised isomorphisms instead of strict isomorphisms.
 Recall a generalised morphism between two proper \'etale   Lie groupoids 
$\mathsf G = (G^1\rightrightarrows G^0)$ and 
$\mathsf H =(H^1 \rightrightarrows H^0)$,
denoted by 
\[
 \xymatrix{ \mathsf G \ar@{-->}[r] & \mathsf H,
 }
 \] 
 is given by a covering 
groupoid $\mathsf G [\cU]$  of $\mathsf G$ associated to an open cover $\cU =\{U_i\}$  of $G^0$, together with 
a strict morphism 
$$\mathsf G [\cU] \longrightarrow \mathsf H.
$$
Here the covering 
groupoid $\mathsf G [\cU]$   is defined to be 
\[
\mathsf G [\cU] = (\bigsqcup_i \mathsf G_{U_i}^{U_j}  \rightrightarrows \bigsqcup_i  U_i )
\]
where $\mathsf G_{U_i}^{U_j} =\{ g\in G^1| s(g) \in U_i, t(g) \in U_j\}$ with the obvious source map and  target map
to $U_i$ and $U_j$ respectively.  There is an obvious strict morphism 
\[
\mathsf G [\cU] \longrightarrow \mathsf G 
\]
which is a strong equivalence. A generalised isomorphism (also called  a Morita equivalence) 
is a generalised morphism   such that the associated  strict
morphism $\mathsf G[ \cU] \to \mathsf H$ is a local isomorphism which induces a homeomorphism
between $|\mathsf G | = |\mathsf G [\cU]|$ and $|\mathsf H|$.  
 
Note that this notion of generalised morphisms and generalised isomorphisms makes sense for 
proper \'etale  topological groupoids. With this understood, then we can proceed to  define a good orbifold gluing structure
as the induction for restriction maps, the induction for gluing maps, the sewing property, and the inward-extendibility condition can be carried over accordingly. Moreover, the proof of Theorem \ref{smooth_structure}
can adapted to get a smooth structure for an orbifold stratified space with a good orbifold gluing structure
in the category of proper \'etale groupoids with morphisms given by generalised morphisms. 
\end{remark}  

As remarked in the introduction,   the disjoint of proper \'etale Lie groupoids
\[
\sfM = \left( \bigsqcup_{\a\in \cS}  M_\a^1 \rightrightarrows  \bigsqcup_{\a\in \cS}  M_\a^0\right) 
\]
often does not admit a topologica groupoid  structure,  as we  shall see 
in the Deligne-Mumford moduli spaces of stables curves.    
The ultimate goal is still to construct a $C^\infty$-compatible orbifold gluing atlas. We need to resolve the issue of both 
\[
\bigsqcup_{\a\in \cS} M^1_\a \ \ \text{and} \ \ \bigsqcup_{\a\in \cS} M^0_\a
\]
have no topological structure. We remark that  orbifold gluing bundles  still make sense, but
gluing maps in Definition \ref{orb_gluing_datum},  such as 
\[\xymatrix{
\phi^\a:  \sfGl^\a(\epsilon)|_{\mathsf U} \ar@{-->}[r] &  \mathsf M^a, }
\]
don't make sense as  $ \mathsf M^a$ is not a topological groupoid. We point out that  the  stratum-wise gluing map
 \[
 \xymatrix{
\phi^\a_\b:  \sfGl^a_\b(\epsilon)|_{\mathsf U}  \ar@{-->}[r] &  \mathsf M_\b
}
\]
 as a generalised isomorphism onto its image is well-defined.  We can just treat the gluing 
 map
 \[\xymatrix{
\phi^\a:  \sfGl^\a(\epsilon)|_{\mathsf U} \ar@{-->}[r] &  \mathsf M^a, }
\] 
 as a generalised isomorphism on the level set-theoretical groupoids. That is, a generalised morphism
 from a  proper \'etale Lie groupoid $\mathsf G$ to a groupoid $\mathsf H$  (not necessarily a topologial
 groupoid)
 is given  by a covering groupoid $\mathsf G [\cU]$ and a strict morphism from $\mathsf G [\cU]$
 to $\mathsf H$ as a set-theoretical groupoid. A  generalised isomorphism
 from a  proper \'etale Lie groupoid $\mathsf G$ to a groupoid $\mathsf H$ is a  generalised morphism
  such that the associated  strict
morphism $\mathsf G[ \cU] \to \mathsf H$ is locally bijecitve and induced a bijective map from
the topological space $|\mathsf G[ \cU]$ to a point-set $|\mathsf H|$.  

Then the notion of good orbifold atlases still makes sense as 
  inductions for restriction maps, inductions for stratum-wise  gluing maps,  sewing property,
   the inward-extendibility condition, gluing compatibility condition   still make sense.   This is due to fact that all these notions only involve smooth structures on the domain groupoids such as $\sfGl^a$. 
   We can then  proceed to establish Theorem
  \ref{orb_structure} even though   the disjoint of proper \'etale Lie groupoids $\{\mathsf M_\a\}_{\a\in \cS}$
  is not a topological groupoid.  We still  achieve a smooth orbifold on the orbifold stratified 
  space $M$  with its orbifold  groupoid 
 $
  \bigsqcup_{\a\in \cS} \sfGl^a (\epsilon)  
 $
  obtained from the particular 
  gluing atlas
  \[
\cG = \{ \mc A_\alpha=(\sfM_\alpha, \mfky^\a,
\epsilon_\alpha,\phi^\alpha,\Phi^\alpha)| \a \in \cS\} 
  \]
    We shall explain how this can be done for the Deligne-Mumford moduli spaces of stable curves.

\section{Moduli spaces of stable curves as orbifold  stratified spaces} \ 

This section is mostly a review of moduli spaces of stable curves.  


\subsection{Teichm\"uller space  and moduli space of Riemann surfaces (top stratum)}  We start with the Teichm\"uller space for
genus $g$  Riemann surfaces with $n$-marked points which play a central role in the description of
moduli space of Riemann surfaces.

Denote by $ \Sigma_{g, n}$  a   genus $g$ smooth oriented compact surface $\Sigma$  with  ordered $n$-marked points $\{p_1, p_2, \cdots, p_n\}$.   Given a genus $g$ compact Riemann surface  with   $n$-marked points
  \[
(C, \{x_1, x_2, \cdots, x_n\}),
\]
a Teichm\"uller structure on $(C, \{x_1, x_2, \cdots, x_n\}) $ is  the datum of  the  isotopy class  $[f]$ of an orientation preserving diffeomorphism
\[
f:  (C,  \{x_1, x_2, \cdots, x_n\}) \longrightarrow \Sigma_{g, n} =  (\Sigma, \{p_1, p_2, \cdots, p_n\}),
\]
where the allowable isotopies are those which map $x_i$ to $p_i$ for each $i=1, 2, \cdots, n$.
Two  genus $g$ compact Riemann surfaces  with Teichm\"uller structures
\[
(C, \{x_1, x_2, \cdots, x_n\};  [f]) \qquad \text{\ and\ } \qquad (C', \{x'_1, x'_2, \cdots, x'_n\};  [f'])
\]
are called isomorphic if there is an isomorphism (a biholomorphism preserving the ordered $n$-marked points)
\[
\phi:  (C, \{x_1, x_2, \cdots, x_n\}) \longrightarrow (C', \{x'_1, x'_2, \cdots, x'_n\})
\]
such that $[f'\circ \phi ] =[f]$.  The  Teichm\"uller space $\cT_{g, n} $  of $\Sigma_{g, n}$ is the set of isomorphism classes of genus $g$,  $n$-marked  compact Riemann surface  with
Teichm\"uller structures.  Any orientation preserving
diffeomorphism between two  genus $g$ smooth oriented compact surfaces with $n$-marked points induces
a canonical  identification between their  Teichm\"uller spaces.  This justifies the simplified notation $\cT_{g, n}$.

Let  $\diff(\Sigma_{g, n})$ be the subgroup of  orientation-preserving  diffeomorphism   group $\diff(\Sigma)$ that fix the $n$ marked point and $\diff_0(\Sigma_{g, n} ) $  be  the   identity component of $\diff(\Sigma_{g, n})$.
The mapping class group of $\Sigma_{g, n}$, denoted by $\Mod_{g, n}$, is
the group of all isotopy classes of orientation-preserving diffeomorphisms of  $\Sigma_{g, n}$, that is,
\[
\Mod_{g, n} = \diff(\Sigma_{g, n})/\diff_0(\Sigma_{g, n}).
\]
 The mapping class group acts naturally on $\cT_{g, n} $ given by
\[
[\gamma]\cdot \big[C, \{x_1, x_2, \cdots, x_n\};  [f] \big] =  \big[C, \{x_1, x_2, \cdots, x_n\};  [\gamma \circ f] \big].
\]
The quotient space of  $\cT_{g, n}$ by $\Mod_{g, n} $ is the moduli space    $M_{g, n}$ of
  genus $g$ Riemann surfaces with  $n$-marked points, this follows from the identifications
\begin{itemize}
\item $\cT_{g, n} =   \cJ(\Sigma)   /\diff_0(\Sigma_{g, n}) =   \big( \cJ(\Sigma) \times  (\Sigma^n \backslash \Delta ) \big) /\diff_0(\Sigma),$
\item $M_{g, n} =  \cJ(\Sigma)   /\diff (\Sigma_{g, n})  = \big( \cJ(\Sigma) \times  (\Sigma^n \backslash \Delta ) \big) /\diff(\Sigma).$
\end{itemize}
Here $\cJ(\Sigma)$ is the space of complex structure on $\Sigma$, and $\Delta$ is the big diagonal so that
$\Sigma^n \backslash \Delta$ is the sub-manifold of $\Sigma^n$ consisting of $n$-distinct points of $\Sigma$.

In \cite{EE69} and \cite{RoSa}, it was showed  that
\[
\xymatrix{
\diff_0(\Sigma) \ar[r] &\cT(\Sigma)  \times  (\Sigma^n \backslash \Delta ) \ar[d]\\
& \cT_{g, n}
}
\]
is a principal fiber bundle. The  associated fiber bundle
\[
\pi_{g, n}:  \big( \cJ(\Sigma) \times  (\Sigma^n \backslash \Delta ) \big) \times_{\diff_0(\Sigma)}  \Sigma \longrightarrow \cT_{g, n}
\]
  for the action of $\diff_0(\Sigma)$ on $\Sigma$ is
a fiber bundle with  fibers diffeomorphic to $\Sigma$ and  having $n$ distinguished  (disjoint) sections.  This fiber bundle is  the
universal curve of genus $g$ with $n$ marked points, will be simply  denoted by
\[
\xymatrix{
\cC_{g, n} \ar[d]_{\pi_{g, n}}  & \\
\cT_{g, n}   \ar@/_1pc/[u]_{ \sigma_{i}, \  i=1, 2, \cdots, n.}
}
\]
 
 The mapping class group $\Mod_{g, n}$  acts on  $ \cT_{g, n}$ as  a properly discontinuous group of
holomorphic transformations. This  action defines an  orbifold structure on $M_{g, n}$.  In terms of
proper \'etale groupoids, this orbifold structure is defined by
the action groupoid
\[
 \cT_{g, n} \rtimes \Mod_{g, n} \rightrightarrows \cT_{g,n}.
\]

In general,  given a family of   genus $g$,  $n$-marked Riemann surfaces $\pi: \cC\to B$, there is a canonical construction to get a proper \'etale Lie groupoid whose unit space is $B$. When $B$ is  the
Teichm\"uller space $\cT_{g, n}$,  then the resultant groupoid is exactly the above action groupoid. 
As this construction is very useful in practics. We devote the next subsection to this construction.

\subsection{Canonical construction of proper 'etale groupoids}\label{sec:groupoidization}
Consider a smooth family of genus $g$,  $n$-marked Riemann surface 
\[
\xymatrix{
\cC  \ar[d]_{\pi}  & \\
B   \ar@/_1pc/[u]_{ \sigma_{i}, \  i=1, 2, \cdots, n . }
}
\]
consisting a smooth fiber bundle $\pi: \cC\to B$ and $n$-disjoint sections $ \sigma_i, i=1, 2, \cdots, n$.  Denote by
\[
(C_b =\pi^{-1}(b), \{\sigma_1(b), \cdots, \sigma_n(b)\}, [f_b])
\]
 the genus $g$ $n$-marked  Riemann surface with Teichm\"uller  structure   for each $b\in B$.  We can construct a
  proper \'etale Lie   groupoid  
 \[
 \mathsf{B}= ( G^1 \rightrightarrows G^0 ) =  (G^1,  G^0, s, t,m,  u, i)
\]
where $G^0 = B$, and $G^1  $ consists of triples
\[
(b_1, \psi, b_2),
\]
for  $b_1, b_2 \in B$ and
 \[
 \psi \in  \text{Isom} \big( (C_{b_2},   \{\sigma_1(b_2), \cdots, \sigma_n(b_2)\}), (C_{b_1},   \{\sigma_1(b_1), \cdots, \sigma_n(b_1)\})\big).
 \]
 The source and tail map $(s,  t)$ are given by the obvious projections
 \[
 s(b_1, \psi, b_2) = b_2, t(b_1, \psi, b_2) = b_1.
 \]
   For any composable pair
 \[
 (b_1, \psi, b_2) \qquad \text{and} \qquad (b_2, \phi, b_3),
 \]
 the groupoid multiplication is defined
 \[
 m ((b_1, \psi, b_2),   (b_2, \phi, b_3) ) =   (b_1, \psi, b_2)\cdot  (b_2, \phi, b_3) = (b_1,   \psi \circ \phi, b_3).
 \]
 The inverse map  $i$  and the unit map $u$ are given by
 \[
i (b_1, \psi, b_2) =  (b_2, \psi^{-1}, b_1)   \qquad \text{and} \qquad  u(b) =  (b, Id, b)
\]
respectively.  To be consistent with our notations, we denote this groupoid by
\begin{equation}
\mathsf B=(\mc A_{\mc C}(B)\rightrightarrows B).
\end{equation}
Here $\mc A_{\mc C}(B)$
 denote the space of arrows among $B$ that are 
 generated from the family $\mc C$. 
\def \wtA{\widetilde{\mathcal A}}

Similarly, we can get a proper \'etale Lie groupoid for  the family $\mc C\to B$ $$
\mathsf C=(\mc A_{\mc C}(\mc C)\rightrightarrows\mc C).
$$
Here $\mc A_{\mc C}
(\mc C)$ is a fibration over $\mc A(\mc C)$: the fiber over $(b_1,\psi,b_2)$ consists 
of triple $(x_1,\psi, x_2)$ where $x_1\in\pi^{-1}(b_1)$ and $x_2=\psi(x_1)$. There is a groupoid  fibration,
 \[
\xymatrix{
\mc A_{\mc C}(\mc C) \ar@<.5ex>[d]\ar@<-.5ex>[d] \ar[r]^{\pi_1} &  \mc A_{\mc C}(B) \ar@<.5ex>[d]\ar@<-.5ex>[d]\\
\mc C \ar[r]_{\pi_0} & B. }
\]
We remark that $\mathsf C$ is the action groupoid associated to the canonical
action of $\mc A_{\mc C}(B) $ on $\mc C$. 

Applying this construction  to the universal family $\mc C_{g,n}\to
\mc T_{g,n}$, we have the groupoid for $M_{g,n}$, i.e, 
$$
\sfM_{g,n}=(\mc A_{\mc C_{g,n}}(\mc T_{g,n}))\rightrightarrows \mc T_{g,n}).
$$
In fact, one can show that 
$$
\mc A_{\mc C_{g,n}}(\mc T_{g,n})\cong \mc T_{g,n}\rtimes \mathrm{Mod}_{g,n}
$$
Then we also have a universal family
$$
\mathsf C_{g,n}=(\mc A_{\mc C_{g,n}}(\mc C_{g,n})\rightrightarrows \mc C_{g,n}),
$$ 
where $\mc A_{\mc C_{g,n}}(\mc C_{g,n}) \cong \mc C_{g,n}\rtimes \mathrm{Mod}_{g,n}$.  

\begin{remark}
Given a smooth family $\pi:\mc C\to  B$, we can certainly have a map $B\to |\sfM_{g,n}|$ or a map
$|\mathsf B|\to |\sfM_{g,n}|$. However we usually do not have  a smooth map from $B$ to $\cT_{g,n}$, hence, the strict morphism between $\mathsf B$ and $\sfM_{g,n}$.  
Instead,
 we  have a  generalised   morphism    
 \[
 \xymatrix{ \mathsf B\ar@{-->}[r] & \sfM_{g,n},
 }
 \]
 in the sense of Remark \ref{HS:morphism}. We briefly review this construction.   Let $b$ be any point in $B$ then there exists a small neighborhood
$U_b$ of  $b$ so that we have a morphism of families
\ba\label{local:family}
    \xymatrix{
\cC|_{U_b} \ar[d]^{\pi} \ar[rr]^{\Phi_b} & &   \cC_{g, n} \ar[d]_{\pi_{g, n}}  & \\
U_b\ar[rr]^{\phi_b}   & & \cT_{g, n}.  
}
\na
Let 
$$
B'=\bigsqcup_{b\in B} U_b,\;\;\;\mc C'=\bigsqcup_{b\in B}\mc C|_{U_b}.
$$
Apply the above canonical groupoid construction to the family $\cC'\to B'$, we get  a proper \'etale Lie groupoid  
\[
\mathsf B' =   (\mc A_{\mc C'}(B')\rightrightarrows B').
\]
It is easy to see that the following diagram 
 \[
\xymatrix{
\mc A_{\mc C'}(\mc B') \ar@<.5ex>[d]\ar@<-.5ex>[d] \ar[r]^{\pi_1} &  \mc A_{\mc C}(B) \ar@<.5ex>[d]\ar@<-.5ex>[d]\\
 B' \ar[r]_{\pi_0} & B  }
\]
define a strong equivalence of Lie groupoids. A strict morphism
\[
\mathsf B' \longrightarrow \mathsf M_{g, n}
\]
can be obtained from the morphism of families in (\ref{local:family}). 
\end{remark}
\subsection{Moduli space of stable curves} \

The  moduli space $M_{g, n}$ is not compact.  It was shown in \cite{DM69} \cite{Knu} that   $M_{g, n}$  can be compactified by adding certain    genus $g$   curves with  $n$-marked points and nodal points.   This compactification is  called  the Deligne-Mumford compactification.

\begin{definition} A   stable curve  $C$ with $n$-marked points $x_1, x_2, \cdots, x_n$ is a  connected
compact complex algebraic curve satisfying the following conditions.
\begin{enumerate}
\item The only non-smooth points are nodal points, locally modelled on the origin in
\[
\{(z_1, z_2) \in \C^2|  z_1 z_2 =0\}.
\]
\item The marked points are distinct smooth points.
\item  The automorphism group of $(C, \{x_1, x_2, \cdots, x_n\})$ is a finite group.
\end{enumerate}
\end{definition}

\begin{remark} In order to better understand the definition of stable curves, a few remarks are needed.
 \begin{enumerate}
\item Topologically, the neighbourhood of a nodal point is homeomorphic to a union of two
discs with their centers identified. So removing a nodal point locally  gives rise to  two
discs with their centers deleted.

\item A nodal point can   be  {\em smoothened}  by replacing two  discs with joint centres
by a cylinder. If all nodal points  in a stable curve $C$ are  smoothened, then the resulting surface is connected. The genus of the resulting surface is called the {\bf genus}  of a stable curve   $C$.

\item  A nodal point is {\bf normalized}  if two discs with joint centres are replaced by disjoint discs.   The {\bf normalization}  of a stable curve $C $ is the smooth
curve obtained from $C$  by normalizing all its nodal points, equivalently,
  the normalization of a stable curve $C $ with the  finite set $\tau$  of nodal points  is a compact smooth Riemann surface $\Sigma = \bigsqcup_v \Sigma_v$ together with a map
  \[
  f: \Sigma \to C
  \]
  such that
  \begin{enumerate}
\item[(i)] $f:  \Sigma \backslash f^{-1}(\tau) \longrightarrow C\backslash \tau$ is biholomorphic.
\item[(ii)] For each nodal point $z\in \tau$, $f^{-1}(z)$ consists of two point.
\end{enumerate}
Let $\Sigma =  \bigsqcup_v \Sigma_v$ be the  normalization of a stable curve $C $.  The image of $\Sigma_v$ under $f$ will be called an {\bf irreducible component}  of $C$.  Each  component $\Sigma_v$ is a smooth Riemann surface with special points consisting of  ordered marked points
\[\Sigma_v \cap f^{-1} ( \{ x_1, x_2, \cdots, x_n\})
\]
and unordered marked points (preimages of the nodes on $\Sigma_v$).    The number of special points on $\Sigma_v$ is denoted by $m_v$.
  Then
\[
|Aut (C, \{x_1, x_2, \cdots, x_n\})| <  \infty
\]
if and only if
\[
 2g_v -2 + m_v > 0,
 \]
   for each component $\Sigma_v$ of $\Sigma$,
 where $g_v$ is the genus of $\Sigma_v$.
Note that   there is a short exact sequence of automorphism groups
\[
1 \to \prod_v Aut(\Sigma_{v, m_v}) \longrightarrow Aut(C,  \{ x_1, x_2, \cdots, x_n\}) \longrightarrow Aut(\Gamma) \to 1,
\]
where $ Aut(\Gamma)$ is the automorphism group of  the weighted dual graph  $\Gamma$ of $C$, and
$Aut(\Sigma_{v, m_v})$ is the subgroup of the automorphism of $\Sigma$ fixing  the  marked   points and the set of unordered marked points.
\end{enumerate}
\end{remark}

\begin{definition}
The coarse moduli space $\overline{M}_{g, n}$  is the  space  of
isomorphism classes of  genus $g$ stable curves with $n$-marked points.   It is the result of
Deligne-Mumford-Knudsen that this space is compact and will be  called the Deligne-Mumford-Knudsen compactification
of the coarse moduli space $M_{g, n}$, denoted by $\overline M_{g, n}$.
\end{definition}

\begin{remark}
 In fact, it is now well-known that $\overline M_{g, n}$ has  a compact complex orbifold structure. We shall denote   the resulting orbifold by $\overline \sfM_{g, n}$.   The proof of this fact  requires a
 construction of (local) universal curves over $\overline M_{g, n}$, see \cite{DM69} \cite{Knu} and \cite{RoSa}.  In the remaining part of this
 paper, we instead apply the gluing principle developed in Section \ref{sec:2}  to provide an orbifold atlas
 on $\overline M_{g, n}$. 
 The main analysis is to show that
 $\overline M_{g, n}$ admits
 a good orbifold  gluing atlas as in Definition \ref{good-orbi:gluing_atlas}.
  Note that  Fukaya and Ono outlined a differential geometric way to endow $\overline M_{g, n}$ with a
 complex orbifold atlas in \cite{FO99}. What we have done below in some sense is to provide a complete detailed gluing theory
 for  $\overline M_{g, n}$  as outline in \cite{FO99}.  
 \end{remark}

Considering  a genus $g$ stable curve  $C$  with $n$-marked points,  its
 topological types  is   classified by the weighted dual graph which we now review.

\begin{definition} A {\bf weighted  dual graph}  $\Gamma$ is a connected graph  together with the assignment of a nonnegative integer weight  to each vertex, denoted by
\[
(V(\Gamma), E(\Gamma), T(\Gamma), g:V(\Gamma) \to \Z_{\geq0},  \ell:   T(\Gamma) \to \{1, 2, \cdots, n\})
\]
where
\begin{itemize}
\item $V(\Gamma)$ is a finite nonempty set of vertices with a weighted function $g: V(\Gamma) \to \Z_{\geq0}$
assigning  a nonnegative
integer $ g_v$  to each vertex $v$.
\item $E(\Gamma)$ is a finite set of edges.
\item $T(\Gamma)$ is a finite set of  $n$-labelled tails with a partition   indexed by $V(\Gamma)$, that is, $T(\Gamma) = \bigsqcup_{v\in V(\Gamma)} T_v$ and   the labelling is  given by
a bijiective  map:  $\ell:   T(\Gamma) \to \{1, 2, \cdots, n\} $.
\end{itemize}
The {\bf genus } of a  weighted dual graph $\Gamma$  is defined to be
\[
g(\Gamma) = \sum_{v\in V(\Gamma)}  g_v + b_1(\Gamma)
\]
where $ b_1(\Gamma)$ is  the first Betti number of the graph $\Gamma$.  A graph $\Gamma$ is called   {\bf
stable}
if for every $v\in V(\Gamma)$,
\ba\label{negative:Euler}
2- 2g_v   -  m_v  <  0
\na
where $m_v$ denotes the valence of  $\Gamma$ at $v$,   the sum of the number of legs  attached to $v$ (cf. Remark \ref{leg}).
\begin{remark}\label{leg}
An edge consists of two half-edges. 
By a leg of $\Gamma$ we mean either a tail or a half-edge.
\end{remark}
Two weighted stable genus $g$ dual graph with $n$-labelled tails
$\Gamma_1$ and $\Gamma_2$ are called isomorphic if there exists a bijection between their
vertices, edges and tails respecting all the relevant structure.
 Denote by  $\cS_{g, n}$  the set  of   isomorphism classes of  weighted stable genus $g$ dual graph with $n$-labelled tails. For a weighted graph $\Gamma$ we denote the class  by $[\Gamma]$.
\end{definition}

Let $\Gamma$ be such a weighted dual  graph. For any edge $e\in E(\Gamma)$ we may contract the edge $e$ to get a new weighted graph $\Gamma'$:
$V(\Gamma')$ and $E(\Gamma')$ are defined in an obvious way; the weight of new vertex is defined such that the genus of $\Gamma'$ is still $g$. Then it is easy to see that $\Gamma'$ is still a stable graph. We may also contract several edges simultaneously. Let $D\subset E(\Gamma)$ be
a subset of edges, then the graph after contracting edges in $D$ is denoted by $Ctr_D(\Gamma)$.

 Given $[\Gamma_1], [\Gamma_2] \in \cS_{g, n}$, we say that $[\Gamma_1]\prec  [\Gamma_2]$   if and only if
 there exist representatives $\Gamma_1$ and $\Gamma_2  $ for  $[\Gamma_1]$ and $[\Gamma_2] $
 respectively, such that $\Gamma_2$ is obtained from a contraction of $\Gamma_1$ along a subset
 of  $E(\Gamma_1)$.   The following lemma is a well known result and we skip the proof.

   \begin{lemma}
 $(\cS_{g, n},\prec)$ is a partially ordered  finite set with a unique top element given by a weighted  dual graph $\Gamma$
 with no edges and  only one vertex of weight $g$ and $n$-labelled tails.
\end{lemma}

 Given  a genus $g$ stable curve  $C$  with $n$-marked points
\[
(C, \{x_1, x_2, \cdots, x_n\}),
\]
we can associate it a weighted  dual graph $\Gamma$  as follows.
There is a vertex  for each irreducible component of $C$ with its weight given by the genus of the component, and its legs
labelled by the marked points on the component, and there is an  edge between a  pair (not necessarily different) of vertices for each nodal point between their components.  One  can check that $\Gamma$ is a  stable weighted  dual graph of genus $g$  with
$n$-labelled legs.

 Given  a   weighted  dual graph  $\Gamma \in \cS_{g, n} $ with $2g-2 +n >0$,  denote  by $M_{\Gamma}$ be the set of  isomorphism classes of genus $g$ stable curve  $C$  with $n$-marked points whose weighted dual graph is  $\Gamma$.   If $\Gamma_1 \cong \Gamma_2$, then  $M_{\Gamma_1} \cong  M_{\Gamma_2} $ set theoretically.   Let   $M_{[\Gamma]}$ be the set of  isomorphism classes of genus $g$ stable curve  $C$  with $n$-marked points whose weighted dual graph   $\Gamma$  belongs to the class $[\Gamma]$.

 The following proposition   is well-known.   We include the proof for the convenience of readers.

\begin{proposition}\label{DM:orbi-stra}
The coarse moduli space $\overline{M}_{g, n}$ has  an orbifold stratified structure
\[
\overline{M}_{g, n} =  \bigsqcup_{[\Gamma] \in \cS_{g, n}} M_{[\Gamma]}
\]
with respect to  $(\cS_{g, n},\prec)$.
\end{proposition}
\begin{proof} The coarse moduli space $\overline M_{g, n}$ is a compact  Hausdorff topological
space.  This can be proved without resorting to the algebraic geometry machinery as in \cite{ACG}. One can extend the the Fenchel Nielsen coordinates to $\overline M_{g, n}$ to show that $\overline M_{g, n}$ is a compact  Hausdorff topological space. See for example  in  \cite{SS} and \cite{EM}.     In \cite{RoSa}, a pure differential geometry proof of this result is provided. 

 Next   we show that each stratum   $M_{[\Gamma]}$ is a smooth orbifold.   Fix a representative $\Gamma$ in the isomorphism class $[\Gamma]\in \cS_{g, n}$.
The normalisation of $\Gamma$ is the new weighted  graph
\[
\tilde \Gamma =\bigsqcup_{v\in V(\Gamma)} \Gamma_v
\]
obtained by severing all the edges in $\Gamma$,   where each connected component $\Gamma_v$ has only one vertex $\{v\}$  of genus $g_v$ with   $p_v$ ordered tails and
$q_v$ half-edges  attached to the vertex $v$.  Note that  $p_v + q_v = val(v)$, the valence of $\Gamma$ at $v$. 

 Denote by
$
 \cT_{g_v, (p_v, q_v)}
$
  the Teichm\"uller space of Riemann surface of genus $g_v$ with
 $p_v$ ordered marked points associated to tails,   and
$q_v$ unordered  marked points associated to half-edges.  Note that 
\[
\sum_{v\in V(\Gamma)} p_v = n, \qquad  \sum_{v\in V(\Gamma)}  q_v = 2 \# E(\Gamma).
\]
On the set of those unordered  marked points for all $v\in V(\Gamma)$, there is
a pairing relation defined by   $E(\Gamma)$. 
 Define 
\[
\cT_{\tilde \Gamma} = \prod_{v\in V(\Gamma)} \cT_{g_v, (p_v, q_v)},
\]
then there is a universal family of curves
\ba\label{uni:Gamma}
\xymatrix{
\cC_{\tilde \Gamma} \ar[d]^{\pi_{\tilde \Gamma}}  & \\
\cT_{\tilde \Gamma}  \ar@/_2pc/[u]_{ \sigma_{i}, \  i=1, 2, \cdots, n  }  \ar@/^2pc/[u]^{\{s^\pm_e |  e\in E(\Gamma)\}  }
}
\na
whose fiber at $([C_v, \{x_1, \cdots, x_{p_v}\}, \{y_1, \cdots, y_{q_v}\}, [f_v]] )_{v\in V(\Gamma)}$ is
 the disjoint union of Riemann surfaces
 \[
 \bigsqcup_{v\in V(\Gamma)}   (C_v, \{x_1, \cdots, x_{p_v}\}, \{y_1, \cdots, y_{q_v}\}).
 \]
 Here the  set of  sections $\{ \sigma_{i}, \  i=1, 2, \cdots, n\}$ is defined by  ordered 
 marked points $$ \bigsqcup_{v\in V(\Gamma)}    \{x_1, \cdots, x_{p_v}\}, $$ and the set of sections
 $\{s^\pm_e | e\in E(\Gamma)\}$ is defined by  paired unordered marked points   $$ \bigsqcup_{v\in V(\Gamma)}    \{y_1, \cdots, y_{q_v}\}. $$  
 
 On the universal family (\ref{uni:Gamma}),  an isomorphism between two fiber curves
 \[
 f:   \bigsqcup_{v\in V(\Gamma)}   (C_v, \{x_1, \cdots, x_{p_v}\}, \{y_1, \cdots, y_{q_v}\}) \longrightarrow 
  \bigsqcup_{v\in V(\Gamma)}   (C'_v, \{x'_1, \cdots, x'_{p_v}\}, \{y'_1, \cdots, y'_{q_v}\})
  \]
  means an isomorphism $f:      \bigsqcup_{v\in V(\Gamma)}   C_v \to  f:   \bigsqcup_{v\in V(\Gamma)}   C'_v$
  which preserves the ordered marked points 
  \[
  f:   \bigsqcup_{v\in V(\Gamma)}    \{x_1, \cdots, x_{p_v}\}  \longrightarrow  \bigsqcup_{v\in V(\Gamma)}    \{x'_1, \cdots, x'_{p_v}\}
  \]
  and  preserves    the  paired sets $\bigsqcup_{v\in V(\Gamma)}    \{y_1, \cdots, y_{q_v}\}$
  and $\bigsqcup_{v\in V(\Gamma)}    \{y'_1, \cdots, y'_{q_v}\}$. 

By identifying $s^+_e$ and $s^-_e$ in $\mc C_{\tilde\Gamma}$ we have a new family of
nodal curves $\mc C_\Gamma$. Set $\mc T_{\Gamma}=\cT_{\tilde\Gamma}$. Then we have a 
universal family 
\ba\label{uni:Gamma'}
\xymatrix{
\cC_{ \Gamma} \ar[d]^{\pi_{\tilde \Gamma}}  & \\
\cT_{ \Gamma}  \ar@/_2pc/[u]_{ \sigma_{i}, \  i=1, 2, \cdots, n  }  \ar@/^2pc/[u]^{\{s_e |  e\in E(\Gamma)\}  }
}.
\na

Applying  the groupoid  construction  in \S\ref{sec:groupoidization}, there is a canonical proper \'etale groupoid 
$$
\sfM_{ [\Gamma]}=(\mc A_{\mc C_{\Gamma}}{(\mc T_{\Gamma})}\rightrightarrows \mc T_{\Gamma})
$$ associated to the universal  family
 (\ref{uni:Gamma'}).  The orbit space of  $\sfM_{[ \Gamma]}$ is $M_{[\Gamma]}$, the set of  isomorphism classes of genus $g$ stable curve  $C$  with $n$-marked points whose weighted dual graph is  $\Gamma$. 
 
Further it is known  that the closure of $M_{[\Gamma]}$ in  the coarse moduli space $\overline{M}_{g, n}$ is given by
\[
\overline{M}_{[\Gamma]} = \bigsqcup_{[\Gamma'] \preceq [\Gamma], [\Gamma']\in \cS_{g, n}} M_{[\Gamma']}.
\]
This says that $\overline{M}_{g,n}$ is an orbifold  stratified space.
 \end{proof}

\begin{remark}\label{nodal:universal curve}
 Associated to the universal family $\mc C_{ \Gamma}\to \mc T_{\Gamma}$. We have  a proper 
 \'etale Lie groupoid description for
the   universal curve $\mc C_\Gamma$
\[
\mathsf C_{[\Gamma]}=(\mc A_{\mc C_{\Gamma}}(\mc C_{\Gamma})
\rightrightarrows  \mc C_{\Gamma})
\]
There is a natural submersion $\mathsf C_{[\Gamma]}\to \sfM_{ [\Gamma]}$. 
 \end{remark}

\begin{remark}\label{not_global_groupoid} Let 
$ M^1 \rightrightarrows M^0 $ be the disjoint of the proper \'etale Lie groupoids over $S_{g, n}$.  Then there is no sensible toplogy of $M^0$  ad $M^0$ such that $ M^1 \rightrightarrows M^0 $  
is a topological groupoid. 
  Hence, this does not fit  with our assumption in  \S\ref{2.4}. However, we shall explain how the notion of good orbifold atlases can still be found following the remarks at the end of Section 2. 
\end{remark}

\section{Horocycle structures associated to marked  or nodal points}

 In this section, we  introduce a notion of horocycle structures and show that each stratum in $\overline M_{g, n}$, there exists a smooth family of horocycle structures in the orbifold sense.

 Let 
 $(C, \{p_1, \cdots, p_n\})$ be a Riemann surface  of genus $g$ and $n$-marked points. Equivalently, we may consider the punctured
 surface of $C$ with marked points removed. 
 $$
 C^*=C\setminus\{p_1,\ldots,p_n\}.
 $$ 
 When $C^*$ is of  negative
 Euler characteristic , $C^*$ has  a complete hyperbolic metric
 $\rho$ (a complete metric of constant curvature
 $-1$)  under which its punctures become   cusps
 of the hyperbolic metric. Locally, the geometry of these cusps can be described by applying the
 uniformization theorem  to $C$  at  a puncture  $p$ as  follows ( see \cite{Wol2007}).

Let $\mathbb H=\{\zeta=x+iy|y>0\}$ be the half upper surface with the Poincare metric 
$$
\rho_o(\zeta)=\frac{1}{(Im(\zeta))^2}d\zeta d\bar\zeta.
$$ 
 Let
$$
\mathfrak D=\dfrac{ \{\zeta \in \mathbb H|  Im (\zeta) \geq 1\}}{\zeta \sim \zeta + 1}
$$
be a cylinder, and $\rho_o$ induces a metric on $\mathfrak D$, which is still denoted  by $\rho_o$.

An interesting result is that for any punctured point $p_i$ there exists a
neighborhood $U_i$ of $p_i$ in $C$ such that 
$$
(U_i,\rho)\cong (\mathfrak D,\rho_o),
$$
moreover, all $U_i$'s are disjoint with each other.  Hence, we 
fix an identification
$$
\zeta_i: U_i\to  \mathfrak D.
$$
 Let $z_i=e^{2\pi i\zeta_i}$. 
Then $z_i$ is a local complex
coordinate on $U_i$ with $z_i(p_i)=0$ and
$$
U_i=\{z_i|\ln |z_i|<-2\pi\}.
$$
 For any $c\in (0, e^{-2\pi})$, the circle $|z_i|=c$ is called
a closed horocycle at $p_i$ with
hyperbolic length $\ell(c)=-2\pi/\log c$. We call $(U_i,z_i)$ a horodisc of $p_i$. 
We define a metric $\mfky_{can}^i$
on $T_{p_i}C$ such that 
\begin{equation}\label{can_metric}
\mfky_{can}^i(\frac{\partial}{\partial z_i},
\frac{\partial}{\partial z_i})=1.
\end{equation}
The {\em canonical horodisc }  of $C$ at $p$ of radius $\delta$  is defined to be
\[
\cD^\delta_p  =\{  x\in C |   | z(x) |  <  \delta\}
\]
for $\delta \leq e^{-2\pi}$.  Denote by
\[
T_pC(\delta) =\{z \in T_pC |  |z|_{\mfky_{can}} < \delta\}.
\]
Under the canonical complex coordinate $z$ near $p$, 
$z$ induces a canonical biholomorphic map
\[
z:  \cD^\delta_p   \longrightarrow  T_pC(\delta).
\]
We denote the inverse map by the horo-map:
$$
hor: T_pC(\delta)\to C.
$$
We call the triple $(\mfky_{can}, \delta, hor)$ the canonical
horocycle structure at $p$.

\begin{definition}\label{horo:str} Let $C$ be a genus  $g$ stable curve with  $n$-marked points $\{p_1, \cdots, p_n\}$ and 
nodal points. Let
 $\pi: \tilde C \to C$ be its normalization.   If $p$ is  a marked point, a {\bf horocycle structure}  on $C$ at $p$ is a metric
 $\mfky$ on $T_pC$ with
 a holomorphic  embedding 
\[
h:  T_p C(\delta) \longrightarrow C
\]
for some small $\delta$ such that
$h(0) =p$ and   $(d h)_0$ is the identity operator. We denote the horocycle structure
by the triple $(\mfky,\delta,h)$. The local inverse map of $h$ is called 
the {\bf   horo-coordinate }   associated to the horocycle structure  at $p$. 
 If $p$ is  a   nodal point, a horocycle structure on $C$ at $p$  is given by a pair of
 metrics $\mfky^\pm$, a pair of small constants $\delta^\pm$ and a pair of maps
 \[
 h^{\pm} = \pi \circ h_{p^\pm}: T_{p^\pm}\tilde C(\delta^\pm)  \longrightarrow \tilde C \longrightarrow C,
\]
where $(\mfky^\pm,\delta^\pm, h^\pm)$ are the horocycle structures on $\tilde C$ at $p^\pm =\pi^{-1}(p)$. 
\end{definition}

\begin{remark} Given a   genus  $g$ stable curve $C$ with  $n$-marked points $\{p_1, \cdots, p_n\}$ and 
nodal points, then $(\mfky_{can}, e^{-2\pi}, hor)$ constructed above defines a   horocycle structure at  each marked point and each
nodal point. We refer it as a canonical horocycle structure. A general horocycle structure on $C$  in Definition \ref{horo:str}  differs from this canonical
one by  a small perturbation.   All  these horodiscs provided by a horocycle structure  are mutually disjoint by taking a  fixed small radius $\delta$.
\end{remark}

The following lemma is  obvious from the definition of horocycle structure.

\begin{lemma}\label{glue:horo}
Let $(\mfky_0,\delta_0,h_0)$ and $(\mfky_1,\delta_1,h_1)$ be two horocycle structures on $C$ at point $p$, then there exists a small constant $\delta<\min(\delta_0,\delta_1)$ such that 
$$
((1-t)\mfky_0+t\mfky_1, \delta, (1-t)h_0+th_1)
$$
defines a family of horocycle structures at $p$ for $t\in [0, 1]$. Here, we use the canonical horocycle coordinate at the neighborhood of $p$.
\end{lemma}

Fix a choice function on $\cS_{g, n}$ such that each $[\Gamma] \in  \cS_{g, n}$ is represented by a weighted dual graph $\Gamma$.  Let $\tilde \Gamma$ be the normalization of $\Gamma$, obtained by severing all the edges in
$\Gamma$.  The proper \'etale Lie groupoid  $\mathsf M_{[\Gamma]}$  for $M_{[\Gamma]}$ is obtained by the universal family 
(\ref{uni:Gamma'}) 
 of stable curve of the weighted dual graph given by $\Gamma$ by the canonical construction in Section 
 \ref{sec:groupoidization}.   Recall that this  universal family  
 is obtained from the universal family   (\ref{uni:Gamma}).  
 Note that the universal family  $\cC_\Gamma \to \cT_\Gamma$ comes with sections 
 \[
 \{\sigma_{i} , s_e|  i=1, \cdots, n, e \in E(\Gamma)\}
 \]
defined  by marked points and
 nodal points. 
 Let $\cC^*_\Gamma$  be  the family of punctured Riemann surfaces, given by
  $ \cC_\Gamma$ with all these sections removed. Then 
 there exists a Euclidean metric on the vertical tangent bundle $T^v \cC^*_\Gamma$ of
 $\hat \pi:  \cC^*_\Gamma \to \cT_\Gamma$ such that the restriction of this metric on each fiber 
 $\hat\pi^{-1}(b)$ is the canonical hyperbolic metric.  Moreover, this smooth family of hyperbolic metrics 
 is invariant under the action of the   proper \'etale Lie groupoid  $\mathsf M_{[\Gamma]}$ .

   Let $L_i$ be the complex line over 
 $\cT_\Gamma$ defined by the pull-back of $T^v \cC_{\tilde \Gamma}$ by the section $\sigma_i$ for $i=1, \cdots, n$,  and $L_e^{\pm}$ be the pair of  complex line over 
 $\cT_\Gamma$ defined by the pull-back of $T^v \cC_{\tilde \Gamma}$ by the section $s_e^\pm$ for $e\in E(\Gamma)$. Here $T^v \cC_{\tilde \Gamma}$ is the vertical tangent bundle of the universal family (\ref{uni:Gamma}).    
 The next proposition implies that there is a smooth family of horocycle structures at  each of marked or nodal points of the universal family
 $\cC_\Gamma \to \cT_\Gamma$.
 
 \begin{proposition} \label{canonhol_hor}  The line bundles $\{   L_i,   L^+_e, L^-_e\}$  are complex line bundles over the proper \'etale Lie groupoid  $\mathsf M_{[\Gamma]}$, and will be denoted by $\mathsf L_i$,  $\mathsf L^+_e$ and $\mathsf L^-_e$ accordingly. 
 Let $\mathsf L$ be one of complex line bundles in 
 $\{\mathsf L_i, \mathsf L_e^\pm |    i=1, \cdots, n, e \in E(\Gamma)\}$.
 There exist a canonical  metric on $\mathsf L$ defined by the smooth family of hyperbolic metrics
 on $T^v\cC^*_\Gamma$. Moreover,  for any $\delta \in (0, e^{-2\pi})$,
 the canonical horo-coordinate  associated to these 
 hyperbolic metrics defines a  smooth strict  morphism 
 \[
 \mathsf{hor_{can}}:  \mathsf L (\delta) \longrightarrow \mathsf C_\Gamma. 
 \]
  \end{proposition}
   \begin{proof}
The only nontrivial part is the smoothness of $\mathsf{hor}$. This follows from Lemma 1.1 (iv)  in
\cite{WolpII}) and the invariance of the hyperbolic metrics under the action of the  proper \'etale Lie groupoid  $\mathsf M_{[\Gamma]}$.
\end{proof}
 
\section{Gluing  data and good orbifold gluing structures  for moduli spaces of stable curves}

With these preparations in Sections 4 and 5, we come to the construction of good orbifold structures for the
orbfiold stratified space
\[
\overline M_{g, n} = \bigsqcup_{[\Gamma] \in \cS_{g, k }} M_{[\Gamma]}
\]
with a choice  function on $\cS_{g, n}$, and the canonical proper \'etale Lie groupoid 
$\mathsf M_{[\Gamma]}$ for $M_{[\Gamma]}$. 
 
 We   start  with  the  orbifold gluing bundle over 
  $\Gl^{[\Gamma]} \to \sfM_{[\Gamma]}$.  
Define the orbifold gluing bundle   to be
 $$
 \mathsf{GL}^{[\Gamma]} =\bigoplus_{e\in E(\Gamma)} 
  \mathsf L_{e}^+ \otimes \mathsf L_{e}^- \longrightarrow  \sfM_{[\Gamma]},
$$ 
where $\mathsf L_e^\pm$ is defined in Proposition \ref{canonhol_hor}.  We remark, this bundle, as a proper
\'etale Lie groupoid can be identified with the canonical groupoid associated to $ \sfM_{[\Gamma]}$-action
on the complex vector bundle
\[
GL^{[\Gamma]} = \bigoplus_{e\in E(\Gamma)} 
  L_{e}^+ \otimes   L_{e}^- \longrightarrow  \cT_{\Gamma},
  \]
  as described in Section \ref{sec:groupoidization}.

 Given any   $[\Gamma']\prec [\Gamma]$, let $\mc E([\Gamma],[\Gamma'])$ be the collection of  subsets 
of $E(\Gamma)$ such that the contraction of $\Gamma$ along  each element in
$\mc E([\Gamma],[\Gamma'])$ is isomorphic to
$\Gamma'$. Set
\[
\mathsf{GL}^{[\Gamma]}_{[\Gamma']} = \bigsqcup_{ I  \in  \mc E([\Gamma],[\Gamma'])} 
 \prod_{e\in I  }  \left( \mathsf L_{e}^+ \otimes \mathsf L_{e}^- \right)^\times,
 \]
 where  $\left( \mathsf L_{e}^+ \otimes \mathsf L_{e}^- \right)^x$ is the $\C^\times $-bundle obtained from
$\mathsf L_{e}^+ \otimes \mathsf L_{e}^-$ with the zero section removed.

 \begin{lemma} \label{prop:gluing-bundle}  Denote by $\cS^{[\Gamma]}$   the partially ordered set 
$
\cS^{[\Gamma]} = \{ [\Gamma'] \in \cS_{g, n}|[ \Gamma'] \succeq [\Gamma]\}.
$
 There exists a  canonical  linear stratification  on $\mathsf{GL}^{[\Gamma]}$ with respect to $\cS^{[\Gamma]}$ 
 $$
 \mathsf{GL}^{[\Gamma]}=\bigsqcup_{[\Gamma']\in \mc S^{[\Gamma]}}
 \mathsf{GL}^{[\Gamma]}_{[\Gamma']}
 $$
 so that 
 $\mathsf{GL}^{[\Gamma]}$ is an orbifold gluing bundle  with a canonical metric (Cf.  Condition  A'). 
 \end{lemma}
 \begin{proof}
Notice that the power set
\[
2^{E(\Gamma)} = \bigsqcup_{[\Gamma'] \in \cS^{[\Gamma]}  }  \mc E([\Gamma],[\Gamma'])
\]
and all elements in $ \mc E([\Gamma],[\Gamma'])$ have the same cardinality. The proof of this lemma is 
straightforward. 
\end{proof}

Next we describe the orbifold gluing datum for the orbifold gluing bundle $ \mathsf{GL}^{[\Gamma]}$ over
any open full subgroupoid $\mathsf U_{[\Gamma]}$ of $\mathsf M_{[\Gamma]}$. We remark  that 
we only need to show the stratum-wise gluing map $\phi^{[\Gamma]}_{[\Gamma']}$  being generalised 
isomorphism. We review how the standard grafting construction can be performed using the 
  smooth family of horocycle structures on  the universal family
 $\cC_\Gamma \to \cT_\Gamma$. By Proposition \ref{canonhol_hor}, we assume that the family of horocycle
 structures along the fiber of $\cC_\Gamma$ is invariant under the action of $\mathsf M_{[\Gamma]}$.
 
   Let $\eta $ be a point in $\cT_\Gamma$ represented by  a nodal curve $C_\eta  $ in $\mc C_\Gamma$, whose   normalization at a nodal point $p_e$ for $e\in E(\Gamma)$  is 
      \[
      (C^\pm, p_{e^\pm}). 
      \]
   Using the horocycle structure of $C_\eta$ at  the nodal point $p$,   for 
   $\delta\in (0, e^{-2\pi})$, 
 we have two  holomorphic  maps    
       \[
   h^\pm:   T_{p_{e^\pm}} C^\pm ( \delta) \longrightarrow C^\pm
   \]
   such that $ h^\pm (0) = p_{e^\pm}$ and the differentials $d    h^\pm$ at $0$  are the identity maps. These provide
 complex local coordinates at $p_{e^+}$ and $p_{e^-}$ of $C^+$ and $C^-$ respectively, denoted by $z$ and $w$. 
 Suppose that $C_\eta$ is decomposed to be two parts
 $$
 C^{out}_{\eta}=C \backslash \left( \{x \in C |  |  z(x) |  \leq {\delta}\} \cup \{y \in C  |  |  w(y) | \leq{\delta}\} \right)
 $$
 and 
 $$
 C^{in}_\eta=\left( \{x \in C |  |  z(x) |  \leq {\delta}\} \cup \{y \in C  |  |  w(y) | \leq{\delta}\} \right).$$
 The gluing construction with respect to a horocycle structure $ h^\pm$ on $C$ is given as the following.  
     Let $t\in  T_{p_{e^+}} C^+ \otimes  T_{p_{e^-}} C^-$ such that $0< |t|< \delta^2$. Define
     \[
     C^\ast_{\eta,t }= C \backslash \left( \{x \in C |  |  z(x) |  \leq \dfrac{|t|}{\delta}\} \cup \{y \in C  |  |  w(y) | \leq \dfrac{|t|}{\delta}\} \right).
     \]
     Then $C_{\eta,t }= C^*_{\eta,t}/\sim$ where  $x\sim y$ in $C^*_{\eta,t} $  if and only if
  \ba\label{plumb}   
  \dfrac{|t|}{\delta} <   |  z(x)|  < \delta, \   \dfrac{|t|}{\delta} <  |w(y)| < \delta  \   \  \text{and} \  \ 
   z(x) w(y) = t.
  \na
   Then for $t\neq 0$,  $C_{\eta,t}$ is a smooth Riemann surface with the weight dual graph given by the
   contraction of $\Gamma$ at $e\in E(\Gamma)$.     
   
   Given $I \in  \mc E([\Gamma],[\Gamma'])$, we can perform the above grafting construction on $C_\eta$ 
 simultaneously for $e\in I$. Suppose $I =\{e_1, \cdots, e_k\}  \in \mc E_{\Gamma,[\Gamma']}$. Let $\eta\in \mc T_{\Gamma}$ and 
$C_\eta$ be the representing curve. For $e_i$ let $y_i\in C_\eta$ be the corresponding nodal point and suppose that horodiscs 
associated to $e_i^\pm$ are $|z_i|<\delta$ and $|w_i|<\delta$ respectively. Denote 
    \[
   \vec{t}=  (t_1, t_2, \cdots, t_k) \in  \prod_{e\in I}  (L^+_e\otimes L^-_e)^\times_{\eta}.
\]
 Then for any $\epsilon<\delta^2$ and $|t_i|<\epsilon$ we have
 \ba\label{cut}
 C_{\eta,\vec{t}} ^* = C_\eta \backslash \left(\bigsqcup_{i=1}^k  \{ x\in C_\eta|     |  z_i(x) | 
  \leq \dfrac{|t_i|}{\delta}\} \cup \{y \in C_\eta  |  |  w_i(y) | \leq \dfrac{|t_i|}{\delta}\} \right). 
 \na
 Then the grafting  construction defines a   nodal curve 
 $C_{\eta,\vec{t}} =  C_{\eta,\vec{t}} ^*/\sim$ where 
  $x\sim y$ in $C^*_{\eta,\vec{t}} $  if and only if for some $i$ 
\[ 
  \dfrac{|t_i|}{\delta} <   |  z_i(x)|  < \delta, \   \dfrac{|t_i|}{\delta} <  |w_i(y)| < \delta  \   \  \text{and} \  \ 
   z_i(x) w_i(y) = t_i.
\]
Therefore, we get 
  a smooth family of nodal curve of type $[\Gamma']$ parametrised by
  $\vec{t}$ in the fiber of 
  \[
  GL^{[\Gamma]}_{[\Gamma']} (\epsilon)  =  \bigsqcup_{ I  \in  \mc E([\Gamma],[\Gamma'])}  \prod_{e\in I } (L^+_e\otimes L^-_e)^\times   (\epsilon) 
  \]
   at $\eta$ for a sufficiently small $\epsilon$. 
   This grafting construction can  also be performed in a small 
   neighbourhood  $U_\eta$ of $\eta$ in $\cT_\Gamma$ with respect a smooth  family of horocycle structure 
   at  the section $s_e$ of $\cC_\Gamma$ over $U_\eta$.   Note this family of nodal curves  is completely determined by the
   family of horocycle structures on the universal family over  $U_\eta$.  Denote the resulting  family of nodal curves  by
   \[
   \cC_{U_\eta, \Gamma', \epsilon} \longrightarrow  GL^{[\Gamma]}_{[\Gamma']}|_{U_\eta}  (\epsilon)
   \]
   which has the topological type $[\Gamma']\in \cS_{g, n}$.  From the grafting construction, we know that 
   this family of nodal curves is a trivial family away from the grafting regions.

   Note that $U_\eta$ can be chosen such that is invariant under the action of $\mathsf M_{[\Gamma]}$ 
   on $\cT_\Gamma$. By the invariance of horocycle
   structures under the action of $\mathsf M_{[\Gamma]}$,  for a sufficiently small $\epsilon$, we get 
   a generalised morphism 
   \ba\label{DM-local:glue}\xymatrix{
   \phi^{[\Gamma]}_{[\Gamma']}:  \mathsf{GL} ^{[\Gamma]}_{[\Gamma']} (\epsilon)|_{\mathsf U_\eta} 
\ar@{-->}[r] &  \ \mathsf M_{[\Gamma']}  }
   \na
   using the canonical construction of proper \'etale Lie groupoid in Section \ref{sec:groupoidization}.  Here
   $\mathsf U_\eta$ is the open full subgroupoid of $\mathsf M_{[\Gamma]}$ associated to $U_\eta$.
   By Theorem A in \cite{WoWol}, we know that this generalised morphism is a generalised isomorphism onto its image using the real analytic coordinate functions on $\cT_{\Gamma'}$ from the gluing parameters 
   in $GL^{[\Gamma]}_{[\Gamma']} (\epsilon)|_{\mathsf U_\eta}$. 
   
   . 
   In particular, given  a proper open full subgroupoid $\mathsf U_\Gamma$ of $\mathsf M_{[\Gamma]}$, there is
  an orbifold gluing map
    \ba\label{DM:glue}\xymatrix{
   \phi^{[\Gamma]}_{[\Gamma']}:  \mathsf{GL} ^{[\Gamma]}_{[\Gamma']} (\epsilon)|_{\mathsf U_\Gamma } 
  \ar@{-->}[r] &  \ \mathsf M_{[\Gamma']},} 
   \na
 which is a generalized isomorphism   onto its image  (see Remark \ref{HS:morphism}).

 \begin{proposition} [Verification of Condition B'] Given  a proper open full subgroupoid $\mathsf U_\Gamma$ of $\mathsf M_{[\Gamma]}$, 
let   $\mathsf D=\mathsf 
D([\Gamma], [\Gamma'], \epsilon)$   and $\mathsf R=\mathsf 
R([\Gamma], [\Gamma'], \epsilon)$
be  the  domain  and the image  of  the gluing map $ \phi_{[\Gamma']}^{[\Gamma]}$ in (\ref{DM:glue}).   There exists a  stratified bundle map 
\ba\label{bundle:map:gpoid}\xymatrix{
\Phi^{[\Gamma]}_{[\Gamma']}:   N({\mathsf{GL}^{[\Gamma]}_{[\Gamma']}})|_{\mathsf D}
\ar@{-->}[r] &  \ \mathsf{GL}^{[\Gamma']}|_{\mathsf R}}
\na
 in the sense of generalised  morphism, such that the following diagram commutes
\[
  \xymatrix{
N(\Gl^{[\Gamma]}_{[\Gamma']})|_{\mathsf 
D}   \ar@{-->}[rr] ^{\Phi^{[\Gamma]}_{[\Gamma']}} \ar[d]  & &
\Gl^{[\Gamma']}  \ar[d]   \\
\mathsf D \ar@{-->}[rr] ^{\phi^{[\Gamma]}_{[\Gamma']}} &  & \sfM_{[\Gamma']}. }
\]
 Moreover, $\Phi^{[\Gamma]}_{[\Gamma']}$ is   a stratified bundle
isomorphism  onto its image $\Gl^{[\Gamma']}|_{\mathsf R}$ in the sense of generalised isomorphism.
\end{proposition}
\begin{proof}  
From the above construction, we know that we cover $\mathsf U$ by a collection of full subgroupoids of  the form $\mathsf U_{\eta} $ for some $\eta\in \cT_\Gamma$.  Then by the definition of generalised morphisms and generalised isomorphisms, we only need to prove the proposition for
$D $ and $R$  being  the  domain  and the image  of  the gluing map $ \phi_{[\Gamma']}^{[\Gamma]}$ in (\ref{DM-local:glue}).  From the grafting construction, we have  a smooth  family  nodal curves of topological type $\Gamma'$,  denoted by 
\[
\cC (U_\eta, [\Gamma'], \epsilon) \to D(U_\eta,[\Gamma'],\epsilon):=
GL^{[\Gamma]}_{[\Gamma']}(\epsilon)|_{U_\eta}.
\]
 Hence, there is a  strict morphism 
 $\Psi_{[\Gamma']}^{[\Gamma]}: \mathsf 
 C (U_\eta, [ \Gamma'], \epsilon ) \to \mathsf 
 C_{[\Gamma']}$ (the universal family over $\sfM_{\Gamma'}$)  such that the following diagram commutes
\ba\label{local:iso}
 \xymatrix{
 \mathsf C (U_\eta, [\Gamma'], \epsilon ) 
   \ar@{-->}[rr]^{\Psi^{[\Gamma]}_{[\Gamma']}} \ar[d]  & &
\mathsf C_{\Gamma'}  \ar[d]   \\
\mathsf D(U_\eta,  [\Gamma'], \epsilon ) 
 \ar@{-->}[rr]^{\phi^{[\Gamma]}_{[\Gamma']}} &  & \sfM_{[\Gamma']},}
\na
which preserves  marked and nodal points. Note that $\mathsf
C (U_\eta,[ \Gamma'], \epsilon )
 \to \mathsf D$ is a trivial fibration away from the gluing region. From the definitions  of the normal bundle
$N({\Gl^{[\Gamma]}_{[\Gamma']}})$ and the gluing bundle $\Gl^{[\Gamma']} $, we know that there is a 
stratified bundle map $\Phi^{[\Gamma]}_{[\Gamma']}$ such that the diagram in Proposition commutes. It is clear that
 $\Phi^{[\Gamma]}_{[\Gamma']}$ is   a stratified bundle
isomorphism  onto its image.
 \end{proof}

 Let  the gluing map 
\[\xymatrix{
\phi^{[\Gamma]}:  \Gl^{[\Gamma]}(\epsilon)|_{\bU_{[\Gamma]}}  =  \bigsqcup_{[\Gamma']\in \cS_{[\Gamma]}} \Gl_{[\Gamma']}^{[\Gamma]} (\epsilon) |_{\bU_{[\Gamma]}} \ar@{-->}[r]&
 \bigsqcup_{[\Gamma']\in \cS_{[\Gamma]}}   \sfM_{[\Gamma']}  }
  \]
 be the union of $\phi^\Gamma_{\Gamma'}$ for $[\Gamma']\in \cS_{[\Gamma]}$. 
 Hence, we 
  get   an orbifold gluing datum   
  $$
  (\bU_{[\Gamma]},  \mfky^{[\Gamma]},
  \epsilon, \phi^{[\Gamma]}, \{\Phi_{[\Gamma']}^{[\Gamma]}\})
  $$ for any proper subgroupoid  $\bU_{[\Gamma]} $ in the sense of comments following Remark \ref{HS:morphism}.

 \begin{remark}
Suppose $\Gamma'$ is obtained from  $\Gamma$ by the contraction along $I\in \cE([\Gamma], [\Gamma'])$. As the family $\cC (U_\eta, [\Gamma'], \epsilon) \to D(U_\eta,[\Gamma'],\epsilon)$ is trivial away from the grafting regions. For any edge $e$ in  the compliment of $I$ in $E(\Gamma)$, the grafting construction
 carries the horocycle structure along the universal curves over $ \mathsf U_\eta$  at $s_e$ to  
\[
\cC (\mathsf U_\eta, [\Gamma'], \epsilon) \to D(\mathsf U_\eta,[\Gamma'],\epsilon). 
\]
Hence, by the commutative diagram (\ref{local:iso}) and local isomorphism of $\Psi^{[\Gamma]}_{[\Gamma']}$,
we get a smooth family of horocycle structure at the section $s_e$ along the that over  the image of the
orbifold gluing map $\phi^{[\Gamma]}_{[\Gamma']}$.  We emphasize that this horocycle structure is different from the canonical
horocycle structure  induced from $\mathsf C_{\Gamma'}$.  From the convexity property of horocycle structures (Cf.
Lemma \ref{glue:horo}, we know that any convex combination of these two horocycle structure provides a 
family of horocycle structure on the universal curves over the image of the 
orbifold gluing map $\phi^{[\Gamma]}_{[\Gamma']}$. We remark that this obversation is vital in obtaining the 
inward-extendibility condition for the existence of good orbifold gluing structures. 
 \end{remark}
 
Notice that the only non-trivial issue for the existence of good orbifold gluing structures as specified by
Definition \ref{good-orbi:gluing_atlas} is the inward-extendibility condition. We proceed to this final issue.

Let $\mathsf U_1\subset\subset\mathsf  U_2$ be a pair of boundary-type open full subgroupoids of $\mathsf M_{[\Gamma]}$.  Suppose that we have an orbifold  gluing datum 
$$
\mc A=(\mathsf U_2,\mfky,\epsilon,\phi^{[\Gamma]}, \{\Phi^{[\Gamma]}_{[\Gamma']}\})
$$
defined with respect to a family of 
 horocycle structures $(\mfky ,\delta, h )$ on the universal curve over $\mathsf U_2$ associated to each
 $e\in E(\Gamma)$.    Let $(\tilde\mfky ,\tilde\delta,\tilde h )$ be
 another  horocycle structures on $\sfM_{\Gamma}$ (for example, the canonical horocycle structures induced from hyperbolic metrics or induced from the gluing map). Such horocycle structures induce a  gluing datum
 $$
\widetilde{\mc A}=(\widetilde{\mathsf U},\tilde\mfky,\tilde\epsilon,\tilde\phi^{[\Gamma]}, \{\widetilde\Phi^{[\Gamma]}_{[\Gamma']}\}),
$$
 where $\widetilde{\mathsf U}$ can be any proper open  full subgroupoid of $\sfM_{[\Gamma]}$ such that
 \[
 \widetilde{\mathsf U} \cup \mathsf U_1 = \sfM_{[\Gamma]}.
 \]
 We want to sew these  two gluing data together.  This is equivalent to patch the horocycle structures. This can be easily done as follows using the  convexity property of horocycle structures.

Let $\beta_1$ and $\beta_2$ be an orbifold 
 partition of unity subordinated  to the cover  $\{\mathsf U_2,  \widetilde{\mathsf U} \}$ such that $\beta_1\equiv 1$ on $\mathsf U_1$. Then applying Lemma \ref{glue:horo}, we get
a  new horocycle structures along the universal curve over $\mathsf M_{[\Gamma]}$, from which we have 
a new  orbifold  gluing datum
$$
\widehat{\mc A}=(\sfM_{[\Gamma]}, \hat\mfky,\hat{\delta},
\hat{\phi}^{[\Gamma]}, \{\Phi^{[\Gamma]}_{[\Gamma']}\}).
$$
It is easy to see this is an inward-extension of $\mc A$.

Let $\mc {GL}$ be the collection of gluing data that 
given by the grafting construction using  horocycle
structures.  Then we have proved the following proposition. 
\begin{proposition}
$\mc{GL}$ is a good  orbifold gluing structure.
\end{proposition}

As a corollary, we have a collection of orbifold  gluing data
\begin{equation}\label{good_atlas}
\{ \mc A_{[\Gamma]}=(\sfM_{[\Gamma]}, \mfky^{[\Gamma]},
\epsilon^{[\Gamma]},\phi^{[\Gamma]},
\Phi^{[\Gamma]}) | [\Gamma]\in \mc S_{g,n}\}, 
\end{equation}
which are  gluing-compatible.

Following the discussion after Remark \ref{HS:morphism},  we have a smooth  orbifold  structure on the Deligne-Mumford moduli space $\overline M_{g, n}$ given by the canonical  proper \'etale Lie groupoid
\[
\left(
(\mc A_{\overline{\cC}_{g,n}}(\overline{\cT}_{g,n}))\rightrightarrows\overline{\cT}_{g,n}\right) 
\]
using the notation from Section \ref{sec:groupoidization}, where $\overline{\cT}_{g,n} = \bigsqcup_{[\Gamma]\in \mc S_{g,n}}  GL^{[\Gamma]}(\epsilon)$, and 
$
\overline{\cC}_{g,n} = \bigsqcup_{[\Gamma]\in \mc S_{g,n}} \cC_{\Gamma, \epsilon}
$
 with 
$\cC_{\Gamma, \epsilon}$ 
being  the universal curve over $GL^{[\Gamma]}(\epsilon)$ associated to
the orbifold gluing datum 
$ 
\mc A_{[\Gamma]}=(\sfM_{[\Gamma]}, \mfky^{[\Gamma]},
\epsilon^{[\Gamma]},\phi^{[\Gamma]},
\Phi^{[\Gamma]})
$ 
  from (\ref{good_atlas}).

 \vskip .2in

\noindent  
{\bf Acknowledgments.}  This work is  supported by   the Australian Research Council   Grant
  and  the National Natural Science Foundation of China Grant.   

\end{document}